\documentclass[11pt,twoside,english]{article}
\usepackage[latin9]{inputenc}
\usepackage[a4paper]{geometry}
\usepackage{fancyhdr}
\usepackage{babel}
\usepackage{amsmath}
\usepackage{amsthm}
\usepackage{amssymb}
\usepackage{bm}
\usepackage[unicode=true,pdfusetitle,
 bookmarks=true,bookmarksnumbered=false,bookmarksopen=false,
 breaklinks=false,pdfborder={0 0 0},pdfborderstyle={},backref=false,colorlinks=true,linkcolor=blue,citecolor=blue]
 {hyperref}
\hypersetup{
 pdfpagemode=UseNone,pdfstartview=FitH,pdfborderstyle={/S/U/W 1}}
\usepackage[export]{adjustbox}
\DeclareMathOperator{\sign}{sgn}
\makeatletter
\theoremstyle{definition}
\newtheorem{defn}{\protect\definitionname}[section]
\theoremstyle{plain}
\newtheorem{thm}{\protect\theoremname}[section]
\ifx\proof\undefined\
  \newenvironment{proof}[1][\proofname]{\par
    \normalfont\topsep6\p@\@plus6\p@\relax
    \trivlist
    \itemindent\parindent
    \item[\hskip\labelsep
          \scshape
      #1]\ignorespaces
  }{%
    \endtrivlist\@endpefalse
  }
  \providecommand{\proofname}{Proof}
\fi
\theoremstyle{plain}
\newtheorem{cor}{\protect\corollaryname}[section]
\newtheorem{observation}{\textbf{Observation}}[section]

\usepackage[small]{caption2}
\usepackage{graphicx}
\usepackage{wrapfig}
\usepackage[scaled]{helvet}

\usepackage{pifont}
\usepackage[dvipsnames]{xcolor}
\usepackage{colortbl}
\newcommand{\cmark}{\textcolor{ForestGreen}{\ding{51}}} 
\newcommand{\xmark}{\textcolor{RedOrange}{\ding{55}}} 

\newcolumntype{M}[1]{>{\centering\arraybackslash}m{#1}}
\setcaptionmargin{0cm}

\renewenvironment{thebibliography}[1]{   
\begin{oldthebibliography}{#1}     
\setlength{\itemsep}{1mm}     
\setlength{\parskip}{0em} } {   
\end{oldthebibliography} }

\makeatother
\providecommand{\corollaryname}{Corollary}
\providecommand{\definitionname}{Definition}
\providecommand{\theoremname}{Theorem}

\begin{document}
\fancyhead[OL]{} 
\fancyhead[OR]{} 
\fancyhead[ER]{}
\fancyhead[EL]{} 
\fancyfoot[EL]{}\fancyfoot[EC]{\thepage} 
\fancyfoot[OR]{}\fancyfoot[OC]{\thepage} 
\renewcommand{\headrulewidth}{0.0pt}%
\title{On Nonlinear Closures for Moment Equations Based on Orthogonal Polynomials}
\author{E.
Yilmaz\thanks{email: \texttt{yilmaz@acom.rwth-aachen.de}},~~G. Oblapenko\thanks{email: \texttt{oblapenko@acom.rwth-aachen.de}},~~and
M. Torrilhon\thanks{email: \texttt{mt@acom.rwth-aachen.de}} \\
Applied and Computational Mathematics, \\
RWTH Aachen Germany }
\date{(2024)}
\maketitle
\begin{abstract}
In the present work, an approach to the moment closure problem on the basis of orthogonal polynomials derived from Gram matrices is proposed. Its properties are studied in the context of the moment closure problem arising in gas kinetic theory, for which the proposed approach is proven to have multiple attractive mathematical properties. Numerical studies are carried out for model gas particle distributions and the approach is compared to other moment closure methods, such as Grad's closure and the maximum-entropy method. The proposed ``Gramian'' closure is shown to provide very accurate results for a wide range of distribution functions. 
\end{abstract}

\section{Introduction}
The moment closure problem, that is, the estimation of a moment of an unknown probability distribution function given the values of the lower-order moments~\cite{kuehn2016moment}, arises in a multitude of scientific contexts, such as gas dynamics and kinetic theory~\cite{levermore1996moment,torrilhon2016modeling}, radiation transport~\cite{frank2006partial}, atmospheric sciences~\cite{milbrandt2005multimoment,yuan2012extended}, and biology~\cite{gillespie2009moment}. It is associated with problems defined by the transport of the underlying probability distribution function, which can be reduced to an evolution equation for its moments. This leads to the logical requirement that the system of evolution equations for the moments of the distribution function, closed via a moment closure, is hyperbolic, that is, corresponds to transport. A multitude of moment closure methods has been developed, such as the approach of Grad~\cite{grad1949kinetic}, quadrature-based approaches~\cite{mcgraw1997description}, projection methods~\cite{koellermeier2014framework,cai2015framework}, and entropy-based methods~\cite{levermore1997entropy,bohmer2020entropic,abdelmalik2023moment} to name a few. However, the question of solving the moment closure problem in a robust, computationally efficient, gauge invariant, and provably globally hyperbolic fashion remains open. As the current work is related to quadrature-based methods for solving the moment closure problem, we discuss them in more detail.

Quadrature-based approaches to the solution of the moment closure problem (hereafter referred to as ``QMOM'') were first introduced in~\cite{mcgraw1997description} and further extended in~\cite{fox2008quadrature,desjardins2008quadrature}. However, they have been found to
result in weakly hyperbolic systems~\cite{chalons2010beyond,fox2018conditional,huang2020stability}, resulting in their limited applicability
to non-equilibrium regimes~\cite{laplante2016comparison,taunay2023quadrature}. Extensions to the method have been sought
that would ensure strict hyperbolicity of the system~\cite{fox2018conditional,van2021higher}, but also provide greater flexibility in terms of computational efficiency and/or accuracy~\cite{fox2023generalized}. Other extensions include the so-called ``extended quadrature method of moments''~\cite{chalons2010multi}, where the solution is sought for as a combination of Gaussian distributions; however, these versions of the quadrature-based approach are not considered in the present work. In~\cite{pichard2020moment,pichard2022moment}, a closure technique based on projection was proposed, with provable weak hyperbolicity and entropy dissipation; this technique can be thought of as a combination of the entropy-based~\cite{levermore1997entropy} and QMOM~\cite{mcgraw1997description,fox2008quadrature} approaches. 

Another related class of closures uses the so-called Kershaw closure~\cite{kershaw1976flux,monreal2008higher,monreal2013moment,schneider2016kershaw} to avoid numerical difficulties at the realizability boundary that arise for example in entropy-base approaches~\cite{alldredge2012high}. These closures investigate the realizability boundary in detail and interpolate between that boundary and the equilibrium point. Typically applied to radiative transfer, a similar interpolation technique was used \cite{mcdonald2013affordable} for gas dynamics.

Recently, in~\cite{FoxLaurent} a novel approach, dubbed ``HyQMOM'', to the moment closure problem was proposed, which does not rely on any reconstruction of the underlying probability distribution function, along with a computational algorithm for the calculation of the required moment.
The present work offers an alternative view on the mathematical properties of the ``HyQMOM'' closure, focusing on structure-preserving aspects such as strict hyperbolicity, gauge invariance, and equilibrium preservation. We present
\begin{itemize}
\item a class of nonlinear hyperbolic closures based on orthogonal polynomials both for the case of an even and odd number of given moments,
\item explicit formulas to calculate the new closures that only require the solution of one or two reduced linear systems,
\item a complete proof of hyperbolicity of the new closures for a general number of moments,
\item assessment of the properties of gauge invariance and equilibrium preservation for the closures,
\item an empirical study of the approximation quality of the new closures in comparison to classical techniques, like the Grad-expansion or maximum-entropy distribution
\item a numerical study of the behavior of the closure near the realizability boundary.
\end{itemize}
The recent pre-print \cite{zhang2024dissipativeness} also investigates hyperbolicity and invariance properties of the theoretical framework of the ``HyQMOM'' closure. However, the methods of the present paper are different and a detailed comparison is left for future work.

The paper is structured as follows.
First, the general moment closure problem is defined in Sec.~\ref{sec:momclosure} and related to kinetic equations. Then, desirable properties which moment closures should preserve are defined and discussed in Sec.~\ref{sec:structurepresmomclosure}.
Next, in Sec.~\ref{subsec:OrthPoly}, the orthogonal polynomial framework is introduced on the basis of Gram matrices. In Sec.~\ref{sec:gramianclosure} the Gramian closure procedure is introduced, and its deficiencies are discussed. On the basis of the Gramian closure, an extended closure, corresponding to the ``HyQMOM'' closure of~\cite{FoxLaurent}, is introduced in Sec.~\ref{sec:extendedGram}, which allows for predicting an odd-order moment based on known moments of up to an even degree. The closure is then proven to possess all the desired mathematical properties.
In Sec.~\ref{sec:The-Odd-Case}, a similar approach is considered, but used for predicting the next even-order moment based on moments of up to an odd degree, and it is shown to not have all of the beneficial properties that the even-order case has. Finally, numerical results are presented in Sec.~\ref{sec:numerical} for several model gas-kinetic distribution functions, and compared with existing approaches, such as Grad's method and the maximum entropy method.

\fancyhead[OL]{Yilmaz, Oblapenko, Torrilhon (2024)} 
\fancyhead[OR]{Nonlinear Closures} 
\fancyhead[ER]{Yilmaz, Oblapenko, Torrilhon (2024)} 
\fancyhead[EL]{Nonlinear Closures} 
\fancyfoot[EL]{}\fancyfoot[EC]{\thepage} 
\fancyfoot[OR]{}\fancyfoot[OC]{\thepage} 
\renewcommand{\headrulewidth}{0.4pt}%

\section{Moment Closures\label{sec:momclosure}}

We consider moments $u_{k}\in\mathbb{R}$ of a distribution $f:\mathbb{R}\to\mathbb{R}^{+},c\mapsto f(c)$
on the real line 
\begin{align}
u_{k} & =\int_{\mathbb{R}}c^{k}f(c)dc
\end{align}
where $f(c)$ satisfies the usual requirements of, e.g., positivity
and integrability. The real line is motivated from the underlying
application, see below, but could also be an interval. The distribution $f$ is supposed to be fixed to a generic
function, which is formally unknown. Only the set of moments up to
order $M\in\mathbb{N}$, that is, the vector $u:=\left\{ u_{k}\right\} _{k=0,1,..,M}\in\mathbb{R}^{M+1},$
is what we know about the distribution. We note, that not all vectors
in $\mathbb{R}^{M+1}$ are realizable as moments of a distribution
function~\cite{hamburger1944hermitian,Shohat1945problem,gautschi1996orthogonal,schmuedgen2017moment}. The subset of realizable moments is denoted
by $\mathcal{R}\subset\mathbb{R}^{M+1}$. 

The main application of moment closures this
paper has in mind is the solution of a 1D kinetic transport
equation
\begin{align}
\partial_{t}f+c\partial_{x}f & =S(f)\quad\quad(x,t,c)\in\Omega\times[0,T]\times\mathbb{R}\label{eq:kineticEqn}
\end{align}
for $f(x,t,c)$ on some spatial interval $\Omega\subset\mathbb{R}$
and with time in $[0,T]$. Here, the distribution space $c\in\mathbb{R}$
represents the velocities of the underlying particle ensemble at any
$(x,t)$. Equation (\ref{eq:kineticEqn}) serves as a model for the
Boltzmann equation in kinetic gas dynamics \cite{Cercignani1988}.
The right hand side $S(f)$ represents the collision operator of the
particles. 

For space-time dependent distributions the moments become functions
of space and time as well. Taking moments of the kinetic equation
leads to a system of partial differential equations in space-time
as evolution for $u_{k}(x,t)$. The system has a hierarchical structure
\begin{align}
\begin{array}{rl}
\partial_{t}u_{0}+\partial_{x}u_{1} & \hspace{-5pt}=s_0\\
\partial_{t}u_{1}+\partial_{x}u_{2}\, & \hspace{-5pt}=s_1\\
\vdots\\
\partial_{t}u_{M}+\partial_{x}u_{M+1} & \hspace{-5pt}=s_M
\end{array}\label{eq:momentSystem}
\end{align}
such that locally in space and time a relation for $u_{M+1}$
is needed to close the system and obtain field equations for the moments
$\{u_{k}\}_{k=0,...,M}$, see also \cite{grad1949kinetic,StruchBook} for the case of gas dynamics. 
In the system (\ref{eq:momentSystem}) the $s_k$ are the moments of the collision 
operator $S(f)$. 
\begin{defn}[Moment Closure]
A prediction of the next moment based on a set of given moments in
the form 
\begin{align}
u_{M+1} & =\mathcal{C}\left(u_{0},u_{1},\cdots, u_{M}\right)\label{eq:closure}
\end{align}
based on a smooth function $\mathcal{C}:\mathcal{R}\to\mathbb{R}$
is called a \emph{moment closure}. 
\end{defn}
Obviously, there is no immediate procedure how the next moment can
be predicted for arbitrary distributions $f$. As a guideline, the closure must exhibit structural properties implied
by the kinetic equation and the physical scenario it represents. Mathematically,
the closure predominantly influences the well-posedness of the solution
of the system (\ref{eq:momentSystem}). In the remainder of the paper
$(x,t)$ will be considered fixed. The role of realizability will be discussed 
in Sec.~\ref{sec:realizboundary}.

\section{Structure-preserving Moment Closures\label{sec:structurepresmomclosure}}

Below we will discuss three aspects which will be relevant for the
nonlinear closures of this paper. If these aspects hold for a specific closure,
we will refer to it as ``structure-preserving closure''.

\subsection{Hyperbolicity}

The moment system (\ref{eq:momentSystem}) is a first order, quasi-linear
system of conservation laws with relaxation
\begin{align}
\partial_{t}U+\partial_{x}F(U) & =S(U)
\end{align}
with a variable vector $U:=\{u_{k}\}_{k=0,...,M}$, a flux function
$F(U):=\{u_{k+1}\}_{k=0,...,M}$ where $u_{M+1}=\mathcal{C}(u_{0},u_{1},\cdots, u_{M})$,
and a algebraic production vector $S(U)$.
It is well-known that the well-posedness of solutions of conservation laws
relies on hyperbolicity, which states that the Jacobian of the flux
function
\begin{align}
DF(U) & =\left(\begin{array}{cccc}
0 & 1\\
 & \ddots & \ddots\\
 &  & 0 & 1\\
\partial_{0}\mathcal{C} & \partial_{1}\mathcal{C} & \cdots & \partial_{M}\mathcal{C}
\end{array}\right)\in\mathbb{R}^{(M+1)\times(M+1)}
\end{align}
has real eigenvalues and is diagonalizable for all realizable moments.
Here we used the abbreviation $\partial_{j}\mathcal{C}:=\partial\mathcal{C}(u_{0},u_{1},\cdots, u_{M})/\partial u_{j}$. 
\begin{defn}[Characteristic Polynomial]
Let $\mathcal{C}$ be a moment closure. The \emph{characteristic
polynomial} of the moment closure is defined by
\begin{align}
\mathcal{P}_{M+1}\left(z;u_{0},u_{1},\cdots, u_{M}\right) & =z^{M+1}-\sum_{j=0}^{M}\partial_{j}\mathcal{C}(u_{0},\cdots, u_{M})\,z^{j}\label{eq:characteristic-polynomial}
\end{align}
where the coefficients depend nonlinearly on the moments through the
closure.
\end{defn}
The form of the characteristic polynomial is directly given by the
moment closure and the hierarchical nature of the moment system. Consequently,
the choice of the closure relation (\ref{eq:closure}) determines
the hyperbolicity of the resulting evolution. 
\begin{defn}[Strict Hyperbolicity]
A moment closure $\mathcal{C}$ is called \emph{strictly hyperbolic},
if $\mathcal{P}_{M+1}\left(z\right)$ has $M+1$ distinct, real-valued
roots for all realizable moments. 
\end{defn}
Distinct roots directly imply the diagonalizability of the Jacobian. For the case of non-strict
hyperbolicity, a lack of eigenvectors of the Jacobian may occur which
leads to delta shocks, as described for example in \cite{mitrovic2007delta}.

\subsection{Gauge-Invariance}

\noindent We consider a gauge transformation for the distribution
in the form
\begin{align}
g(\widetilde{c}) & =\rho^{-1}\theta^{1/2}f(\theta^{1/2}\widetilde{c}-v)\label{eq:gaugeTransform}
\end{align}
using general gauge parameters $\rho,\theta,v\in\mathbb{R}$ with
$\rho,\theta>0$, which represent an arbitrary density, temperature
and velocity. The transformed moments can be calculated by
\begin{align}
\widetilde{u}_{k} & =\int_{\mathbb{R}}\widetilde{c}^{k}g(\widetilde{c})d\widetilde{c}=\rho^{-1}\theta^{1/2}\int_{\mathbb{R}}\widetilde{c}^{k}f(\theta^{1/2}\widetilde{c}-v)d\widetilde{c}=\frac{1}{\rho}\int_{\mathbb{R}}\left(\frac{c+v}{\theta^{1/2}}\right)^{k}f(c)dc\\
 & =\frac{1}{\rho\theta^{k/2}}\int_{\mathbb{R}}\sum_{j=0}^{k}{\textstyle \binom{k}{j}}v^{j}c^{k-j}f(c)dc\nonumber 
\end{align}
which yields a linear combination of the original moments 
\begin{align}
\widetilde{u}_{k}= & \frac{1}{\rho\theta^{k/2}}\sum_{j=0}^{k}{\textstyle \binom{k}{j}}v^{j}u_{k-j}\quad\quad\quad k=0,...M+1.\label{eq:gaugeTransformMoments}
\end{align}
Note that for $(\rho,v,\theta)=(1,0,1)$ we have $\widetilde{c}=c$
and $g(c)=f(c)$ and $\widetilde{u}_{k}=u_{k}$. The kinetic equation
(\ref{eq:kineticEqn}) is gauge-invariant if the parameters are constant
and space-time is additionally transformed as $(x,t)=(\widetilde{x}-v\widetilde{t},\widetilde{t})$.
\begin{defn}[Gauge-Invariant Closure]
\label{def:gauge-invariance}A moment closure $\mathcal{C}$ is called
\emph{gauge-invariant}, if the closure relation does not depend on
the gauge parameters $(\rho,v,\theta)$, that is,
\begin{align}
\widetilde{u}_{M+1}=\mathcal{C}\left(\widetilde{u}_{0},\cdots,\widetilde{u}_{M}\right)\quad\Leftrightarrow & \quad u_{M+1}=\mathcal{C}\left(u_{0},\cdots, u_{M}\right)\label{eq:gaugeInvariantClosure}
\end{align}
for the transformed moments in (\ref{eq:gaugeTransformMoments}). 
\end{defn}
Gauge-invariance includes Galilean invariance with the constant velocity
$v$, and scaling invariance of the velocity variable with the thermodynamic
speed $\sqrt{\theta}$. The density $\rho$ represents a normalization
scale. Some interesting insight is given by the following theorem,
similar to the one deduced in Fox \& Laurent \cite{FoxLaurent} based on the explicit study of the
Jacobian of the moment system (\ref{eq:momentSystem}). It should be noted that in the present proof,
unlike the one in \cite{FoxLaurent}, the gauge parameters are chosen arbitrarily and are not
necessarily related to the moments $u_0$, $u_1$, $u_2$.
The proof
below shows the deep connection to invariance that is independent of any moment
equations.
\begin{thm}
\noindent \label{thm:gauge-invariance}For a moment closure $u_{M+1}=\mathcal{C}\left(u_{0},\cdots, u_{M}\right)$
the following statements are equivalent.\vspace{-2mm}
\end{thm}
\begin{enumerate}
\item[(i)] \textit{The closure is gauge-invariant.}\vspace{-2mm}
\item[(ii)] \textit{For all moments $\{u_{k}\}_{k=0,...M}\in\mathcal{R}$ the
characteristic polynomial $\mathcal{P}_{M+1}\left(z\right)$ satisfies}
\begin{align}
\int_{\mathbb{R}}\mathcal{P}_{M+1}(c)f(c)dc=0,\quad\quad\int_{\mathbb{R}}\mathcal{P}_{M+1}'(c)f(c)dc=0,\quad\quad & \int_{\mathbb{R}}c\mathcal{P}_{M+1}'(c)f(c)dc=0\label{eq:invarianceConditions}
\end{align}
\end{enumerate}
\begin{proof}
\noindent Spelling out the invariant closure relation in (\ref{eq:gaugeInvariantClosure})
we obtain
\begin{align}
\widetilde{u}_{M+1}(\rho,v,\theta;u_{0},\cdots, u_{M+1})-\mathcal{C}\left(\left\{ \widetilde{u}_{k}(\rho,v,\theta;u_{0},\cdots, u_{M})\right\} _{k=0,...M}\right) & =0\label{eq:invariance1}
\end{align}
which should not depend on $(\rho,v,\theta)$, hence the derivative
of this expression with respect to $(\rho,v,\theta)$ is zero for
all moments. The transformation (\ref{eq:gaugeTransformMoments})
can be written as 
\begin{align}
(\widetilde{u}_{0},\widetilde{u}_{1},\cdots,\widetilde{u}_{M+1})^{T} & =Q_{M+1}(\rho,v,\theta)\cdot(u_{0},u_{1},\cdots, u_{M+1})^{T}
\end{align}
with an invertible matrix $Q_{M+1}(\rho,v,\theta)\in\mathbb{R}^{(M+1)\times(M+1)}$.
We define the three derivatives as
\begin{align}
(d_{0}^{(\rho,v,\theta)},d_{1}^{(\rho,v,\theta)},\cdots, d_{M+1}^{(\rho,v,\theta)})^{T} & :=\partial_{(\rho,v,\theta)}Q_{M+1}\cdot(u_{0},u_{1},\cdots, u_{M+1})^{T}\label{eq:derivativesTransform}
\end{align}
where we suppressed the arguments of $Q_{M+1}$. The derivatives of the
closure relation (\ref{eq:invariance1}) are now written in the way
\begin{align}
d_{M+1}^{(\rho,v,\theta)}-\sum_{k=0}^{M}d_{k}^{(\rho,v,\theta)}\partial_{k}\mathcal{C}\left(\widetilde{u}_{0},\cdots,\widetilde{u}_{M}\right) & =0\label{eq:gaugeInvCond1}
\end{align}
which must hold for all moments $u$, or equivalently for all transformed
moments $\widetilde{u}$. We write the derivatives in (\ref{eq:derivativesTransform})
as a function of transformed moments by using
\begin{align}
(d_{0}^{(\rho,v,\theta)},d_{1}^{(\rho,v,\theta)},\cdots, d_{M+1}^{(\rho,v,\theta)})^{T} & =\partial_{(\rho,v,\theta)}Q_{M+1}\cdot Q_{M+1}^{-1}\cdot(\widetilde{u}_{0},\widetilde{u}_{1},\cdots,\widetilde{u}_{M+1})^{T}.
\end{align}
 A short computation shows
\begin{align}
d_{k}^{(\rho)}=-\frac{1}{\rho}\widetilde{u}_{k},\quad d_{k}^{(v)}=\frac{k}{\theta^{1/2}}\widetilde{u}_{k-1},\quad & d_{\theta}^{(v)}=-\frac{k}{2\theta}\widetilde{u}_{k}
\end{align}
and from (\ref{eq:gaugeInvCond1}) we obtain the three relations 
\begin{align}
\widetilde{u}_{M+1}-\sum_{k=0}^{M}\widetilde{u}_{k}\partial_{k}\mathcal{C}\left(\widetilde{u}_{0},\cdots,\widetilde{u}_{M}\right) & =0,\\
(M+1)\widetilde{u}_{M}-\sum_{k=1}^{M}k\,\widetilde{u}_{k-1}\partial_{k}\mathcal{C}\left(\widetilde{u}_{0},\cdots,\widetilde{u}_{M}\right) & =0,\\
(M+1)\widetilde{u}_{M+1}-\sum_{k=1}^{M}k\,\widetilde{u}_{k}\partial_{k}\mathcal{C}\left(\widetilde{u}_{0},\cdots,\widetilde{u}_{M}\right) & =0,
\end{align}
which have to hold for all $\{\widetilde{u}_{k}\}_{k=0,...,M}\in\mathcal{R}$.
Inspection shows that these correspond to the integrals of $\mathcal{P}_{M+1}(c)$,
$\mathcal{P}_{M+1}'(c)$, and $c\mathcal{P}_{M+1}'(c)$ of the characteristic
polynomial (\ref{eq:characteristic-polynomial}).
\end{proof}

\subsection{Equilibrium}
We can define the concept of an equilibrium distribution $f^{(\text{eq})}$ for the collision operator $S(f)$ appearing in (\ref{eq:kineticEqn}) via the requirement
\begin{align}
S(f^{(\text{eq})}) \equiv 0\label{eq:equilibriumDef}.
\end{align}
In the context of classical kinetic gas theory, Gaussian distributions fulfill this definition. 
The equilibrium distributions
are typically written as Maxwell distributions 
\begin{align}
f^{(\text{eq})}(c) & =\frac{\rho}{(2\pi\theta)^{1/2}}\exp\left(-\frac{(c-v)^{2}}{2\theta}\right)\label{eq:maxwellian}
\end{align}
with some density $\rho$, velocity $v$, and temperature $\theta$.
\begin{defn}[Equilibrium Preservation]
A moment closure $\mathcal{C}$ is called \emph{equilibrium-preserving},
if the closure relation reproduces equilibrium moments, that is,
\begin{align}
u_{M+1}^{(\text{eq})}= & \mathcal{C}\left(u_{0}^{(\text{eq})},u_{1}^{(\text{eq})},\cdots, u_{M}^{(\text{eq})}\right)\label{eq:equilibriumPreserv}
\end{align}
where the moments $\{u_{k}^{(\text{eq})}\}_{k=0,...,M+1}$ are based
on any equilibrium distribution.
\end{defn}
Note that the equilibrium distribution is independent of the kinetic transport. In fact, an almost arbitrary family of distributions could be declared an equilibrium. In that sense equilibrium preservation can be considered as an external constraint.

\section{Orthogonal Polynomials\label{subsec:OrthPoly}}

The challenge is to find a closure relation (\ref{eq:closure}) such
that strict hyperbolicity, gauge-invariance, and equilibrium-preservation
can be guaranteed. The recent paper \cite{FoxLaurent} suggests to
consider orthogonal polynomials, which are introduced below.

We denote the space of monovariate polynomials of up to degree $k$ as $\mathbb{P}_{k}$.
The distribution function $f(c)$ implies a scalar product of the
form
\begin{align}
\left\langle v,w\right\rangle _{f} & =\int_{\mathbb{R}}v(c)w(c)f(c)dc
\end{align}
which can be used to define orthogonal polynomials $p_{k}\in\mathbb{P}_{k}$
such that 
\begin{align}
\left\langle p_{k},p_{l}\right\rangle _{f} & =0\quad\text{for }k\ne l.
\label{eq:orthopoly}
\end{align}
As a normalization condition, we set the highest coefficient to unity
so that the polynomials are \emph{monic}. The set of $\{p_{k}\}_{k=0,1,2,...,n}$
can be found by Gram-Schmidt orthogonalization of the monomials $1,c,...,c^{n}$,
but for the purpose of this paper we will rely on Gram matrices.
\begin{defn}[Gramian Matrix]
Let the moments $\{u_{k}\}_{k=0,...,2n} \subset \mathcal{R}$ of a distribution be given and finite. Then the matrix
\begin{align}
G_{n} & =\left\{ \left\langle c^{i},c^{j}\right\rangle _{f}\right\} _{i,j=0,...,n}=\left(\begin{array}{ccccc}
u_{0} & u_{1} & u_{2} & \cdots & u_{n}\\
u_{1} & u_{2} &  &  & \vdots\\
u_{2} &  &  &  & \vdots\\
\vdots &  &  &  & \vdots\\
u_{n} & \cdots & \cdots & \cdots & u_{2n}
\end{array}\right)\in\mathbb{R}^{(n+1)\times(n+1)}\label{eq:gram-matrix}
\end{align}
is called the monomial-based
\emph{Gramian matrix} of order $n\in\mathbb{N}$.
\end{defn}
The Gramian matrix is of Hankel-type, but when originating from a
scalar product these matrices are typically associated with Gram \cite{LinAlgBook}.
It is easy to show that $G_{n}$ is symmetric positive definite for
all $n$ and hence invertible. It is well-known that the moment vector
$\{u_{k}\}_{k=0,...,2n}$ is realizable if $G_{n}$ is positive definite
\cite{hamburger1944hermitian,Shohat1945problem,gautschi1996orthogonal,schmuedgen2017moment}.

The theory of this paper is built around a different fact, namely that 
the Gramian matrix allows to give an
explicit formula for the orthogonal polynomials (\ref{eq:orthopoly}) when based on the same scalar product. 
The explicit representation
can be found in the classical text books \cite[Chapter 2.2]{SzegoeOrthogonal}
and \cite[Chapter 14]{KowalDet}, and is given by
\begin{align}
p_{n}(c)= & \frac{1}{\det G_{n-1}}\det\left(\begin{array}{cc}
G_{n-1} & \begin{array}{c}
u_{n}\\
\vdots\\
u_{2n-1}
\end{array}\\
\begin{array}{cccc}
1, & c, & \cdots, & c^{n-1}\end{array} & c^{n}
\end{array}\right).\label{eq:gram-representation}
\end{align}
This expression is obviously a monic polynomial of degree $n$ and
due to the row-wise linearity of the determinant it can be shown that
$\left\langle c^{j},p_{n}(c)\right\rangle _{f}=0$ for $j=0,1,...,n-1$.
Hence, $p_{n}$ in (\ref{eq:gram-representation}) is the desired
orthogonal polynomial. Using the determinant of the Schur-complement
(\ref{eq:schurdet}) from Appendix \ref{sec:Schur} the formula in
(\ref{eq:gram-representation}) can be further simplified to 
\begin{align}
p_{n}(c) & =c^{n}-(1,c,\cdots, c^{n-1})G_{n-1}^{-1}u_{n,2n-1}\label{eq:pn-explicit}
\end{align}
where we used the abbreviation
\begin{align}
u_{k,l}:=\left(\begin{array}{c}
u_{k}\\
u_{k+1}\\
\vdots\\
u_{l}
\end{array}\right)\in\mathbb{R}^{l-k+1} & \quad\text{for}\quad k\le l
\end{align}
for the vector $u_{n,2n-1}\in\mathbb{R}^{n}$.

\section{Gramian Closure\label{sec:gramianclosure}}

This section will introduce a hyperbolic moment closure based on orthogonal
polynomials different to that of \cite{FoxLaurent}. It allows an
easy proof of hyperbolicity, however, it will not satisfy complete
gauge-invariance. Instead, it will be used as a building block for the
extended closures in the next section. Both sections are based on the
case of even $M$ as the maximum moment degree. The odd case will be discussed
in Sec.~\ref{sec:The-Odd-Case}.

\subsection{Definition and Hyperbolicity}
\begin{defn}[Gramian Closure -- even]
Let $M\in\mathbb{N}$ be even and $n=\frac{M}{2}$. The condition
\begin{align}
\int_{\mathbb{R}}p_{n}(c)c^{n+1}f(c)dc & =0\label{eq:closure-relation}
\end{align}
defines the relation 
\begin{align}
u_{2n+1} & =u_{n+1,2n}^{T}G_{n-1}^{-1}u_{n,2n-1}\label{eq:gramian-closure}
\end{align}
between $u_{M+1}=u_{2n+1}$ and $\left\{ u_{k}\right\} _{k=0,1,..,M}$,
which we call \emph{Gramian moment closure}. 
\end{defn}

Note that this closure in the form \eqref{eq:closure-relation} has very little resemblance 
with quadrature-based closures that represent the distribution using Dirac-deltas.
Instead, we impose a polynomial projection similar to traditional Grad moment methods. 
In fact, the corresponding Grad closure would read
\begin{align}
\int_{\mathbb{R}}\operatorname{He}_{2n+1}(c)f(c)dc & =0\label{eq:closure-grad}
\end{align}
using a Hermite-polynomial $\operatorname{He}_{2n+1}(c)$ based on a local or global
using a Hermite-polynomial $\operatorname{He}_{2n+1}(c)$ based on a local or global
equilibrium function. When extracting a closure relation for $u_{M+1}=u_{2n+1}$
from this, it is straightforward to show that the characteristic polynomial 
is also $\operatorname{He}_{2n+1}(z)$ in case of a global equilibrium. In this 
section and the following we will clarify how the characteristic polynomial
can be deduced explicitly from expressions like \eqref{eq:closure-relation}.

The expression (\ref{eq:gramian-closure}) follows directly from the
integration of the polynomial in (\ref{eq:pn-explicit}). The evaluation
of the closure only requires a linear solve of a system of size $\frac{M}{2}$.
For $M=4$ we obtain the explicit form
\begin{align}
u_{5} & =\left(\begin{array}{c}
u_{3}\\
u_{4}
\end{array}\right)^{T}\left(\begin{array}{cc}
u_{0} & u_{1}\\
u_{1} & u_{2}
\end{array}\right)^{-1}\left(\begin{array}{c}
u_{2}\\
u_{3}
\end{array}\right),
\end{align}
while for $M=6$ we find
\begin{align}
u_{7} & =\left(\begin{array}{c}
u_{4}\\
u_{5}\\
u_{6}
\end{array}\right)^{T}\left(\begin{array}{ccc}
u_{0} & u_{1} & u_{2}\\
u_{1} & u_{2} & u_{3}\\
u_{2} & u_{3} & u_{4}
\end{array}\right)^{-1}\left(\begin{array}{c}
u_{3}\\
u_{4}\\
u_{5}
\end{array}\right).
\end{align}
In principle, the explicit form allows to compute the partial derivatives
with respect to the moments and write down the characteristic polynomial
(\ref{eq:characteristic-polynomial}). 
\begin{thm}
\label{thm:The-Gramian-closure}The Gramian closure (\ref{eq:closure-relation})
implies the characteristic polynomial
\begin{align}
\mathcal{P}_{M+1}^{(\text{G})}\left(z\right) & =p_{n}(z)p_{n+1}(z),\label{eq:gram-charac}
\end{align}
 where $p_{n+1}$ uses the moment $u_{2n+1}$ implied by the closure
(\ref{eq:gramian-closure}). 
\end{thm}
\begin{proof}
We will directly compare the coefficients in the product of the two
orthogonal polynomials with the partial derivatives of the closure
which make up the characteristic polynomial according to (\ref{eq:characteristic-polynomial}).
We have 
\begin{align}
p_{n}(c) & =c^{n}-(1,c,\cdots, c^{n-1})G_{n-1}^{-1}u_{n,2n-1},\\
p_{n+1}(c) & =c^{n+1}-(1,c,\cdots, c^{n})G_{n}^{-1}u_{n+1,2n+1}\label{eq:pn+1}
\end{align}
for the two orthogonal polynomials. To simplify the expression for
$p_{n+1}$ we will use the Schur complement, see Appendix~\ref{sec:Schur}.
The Gramian $G_{n}$ is written in the form
\begin{align}
G_{n} & =\left(\begin{array}{cc}
G_{n-1} & \begin{array}{c}
u_{n}\\
\vdots\\
u_{2n-1}
\end{array}\\
\begin{array}{ccc}
u_{n}, & \cdots &, u_{2n-1}\end{array} & u_{2n}
\end{array}\right)
\end{align}
and for the Schur-complement (\ref{eq:schur}) we identify $A=G_{n-1}\in\mathbb{R}^{n\times n}$,
$B=C=u_{n,2n-1}\in\mathbb{R}^{n}$ and $D=u_{2n}\in\mathbb{R}$, such
that $k=n$ and $l=1.$ We find for the bilinear form in the polynomial
$p_{n+1}$ given in (\ref{eq:pn+1})
\begin{align}
 & (1,c,\cdots, c^{n})G_{n}^{-1}u_{n+1,2n+1}=\\
 & (1,c,\cdots, c^{n-1})G_{n-1}^{-1}u_{n+1,2n}+\frac{(c^{n}-(1,c,\cdots, c^{n-1})G_{n-1}^{-1}u_{n,2n-1})(u_{2n+1}-u_{n,2n-1}^{T}G_{n-1}^{-1}u_{n+1,2n})}{u_{2n}-u_{n,2n-1}^{T}G_{n-1}^{-1}u_{n,2n-1}}\nonumber
\end{align}
according to formula (\ref{eq:schurinv}) which yields
\begin{align}
p_{n+1}(c) & =c^{n+1}-(1,c,\cdots, c^{n-1})G_{n-1}^{-1}u_{n+1,2n}-\frac{u_{2n+1}-u_{n,2n-1}^{T}G_{n-1}^{-1}u_{n+1,2n}}{u_{2n}-u_{n,2n-1}^{T}G_{n-1}^{-1}u_{n,2n-1}}p_{n}(c).\label{eq:pn+1-explicit}
\end{align}
Here we need to use the Gramian closure (\ref{eq:gramian-closure})
for the moment $u_{2n+1}$ which makes the fraction term vanish. Evaluating
the product
\begin{align}
p_{n}(c)p_{n+1}(c)= & \left(c^{n}-(1,c,\cdots, c^{n-1})G_{n-1}^{-1}u_{n,2n-1}\right)\left(c^{n+1}-(1,c,\cdots, c^{n-1})G_{n-1}^{-1}u_{n+1,2n}\right)\label{eq:orthoProduct}
\end{align}
leads to three individual polynomial terms
\begin{align}
p_{n}(c)p_{n+1}(c)= & \,\,c^{2n+1}\label{eq:explicitproduct}\\
 & -(c^{n+1},c^{n+2},\cdots, c^{2n})G_{n-1}^{-1}u_{n,2n-1}\nonumber \\
 & -(c^{n},c^{n+1},\cdots, c^{2n-1})G_{n-1}^{-1}u_{n+1,2n}\nonumber \\
 & +u_{n,2n-1}^{T}G_{n-1}^{-1}(1,c,\cdots, c^{n-1})^{T}(1,c,\cdots, c^{n-1})G_{n-1}^{-1}u_{n+1,2n}.\nonumber 
\end{align}
On the other hand the derivatives of the closure that build the characteristic
polynomial (\ref{eq:characteristic-polynomial}) also come with three
terms
\begin{align}
\partial_{u_{k}}\mathcal{C}\left(u_{0},u_{1},\cdots, u_{2n+1}\right)=\partial_{u_{k}}\left(u_{n,2n-1}^{T}G_{n-1}^{-1}u_{n+1,2n}\right) & =\partial_{u_{k}}(u_{n,2n-1}^{T})G_{n-1}^{-1}u_{n+1,2n}\label{eq:derivativeC}\\
 & +\partial_{u_{k}}(u_{n+1,2n}^{T})G_{n-1}^{-1}u_{n,2n-1}\nonumber \\
 & -u_{n,2n-1}^{T}G_{n-1}^{-1}\partial_{u_{k}}(G_{n-1})G_{n-1}^{-1}u_{n+1,2n}\nonumber 
\end{align}
where we used $\partial_{t}(A^{-1})=-A^{-1}\partial_{t}(A)A^{-1}$
for the derivative of an invertible matrix $A$. It is now easy to
conclude that the derivative with respect to $u_{k}$ in (\ref{eq:derivativeC})
indeed corresponds to the coefficient of $c^{k}$ in the product (\ref{eq:explicitproduct}).
Hence, we have $\mathcal{P}_{2n+1}^{(\text{G})}(c)=p_{n}(c)p_{n+1}(c)$
for all $n\in\mathbb{N}.$
\end{proof}
\begin{cor}
The Gramian closure (\ref{eq:gramian-closure}) is strictly hyperbolic.
\end{cor}
\begin{proof}
All orthogonal polynomials have interlaced and distinct, real roots,
hence, the characteristic polynomial $\mathcal{P}_{M+1}^{(\text{G})}$
in (\ref{eq:gram-charac}) gives strict hyperbolicity. See also the
discussion in Appendix \ref{sec:Interlacing-Theorem}.
\end{proof}

\subsection{Invariance\label{subsec:InvarianceGram}}

It is easy to check the three invariance conditions (\ref{eq:invarianceConditions})
on the characteristic polynomial (\ref{eq:gram-charac}). For the
first condition we immediately obtain
\begin{align}
\int_{\mathbb{R}}\mathcal{P}_{M+1}^{(\text{G})}(c)f(c)dc & =\int_{\mathbb{R}}p_{n}(c)p_{n+1}(c)f(c)dc=0
\end{align}
due to orthogonality. For the third condition we rely on the representation
\begin{align}
p_{n+1}(c) & =c^{n+1}+q_{n-1}(c)
\end{align}
with a polynomial $q_{n-1}\in\mathbb{P}_{n-1}$, which is valid for
the Gramian closure due to (\ref{eq:pn+1-explicit}). We compute
\begin{align}
\int_{\mathbb{R}}c{\textstyle \frac{d}{dc}}\mathcal{P}_{M+1}^{(\text{G})}(c)f(c)dc & =\int_{\mathbb{R}}cp_{n}'(c)p_{n+1}(c)f(c)dc+\int_{\mathbb{R}}p_{n}(c)cp_{n+1}'(c)f(c)dc\nonumber\\
 & =(n+1)\int_{\mathbb{R}}p_{n}(c)c^{n+1}f(c)dc+\int_{\mathbb{R}}p_{n}(c)cq_{n-1}'(c)f(c)dc
\end{align}
where the first term vanishes due to the closure relation and the
second due to orthogonality. However, for the second condition we
find
\begin{align}
\int_{\mathbb{R}}{\textstyle \frac{d}{dc}}\mathcal{P}_{M+1}^{(\text{G})}(c)f(c)dc & =\int_{\mathbb{R}}p_{n}'(c)p_{n+1}(c)f(c)dc+\int_{\mathbb{R}}p_{n}(c)p_{n+1}'(c)f(c)dc\nonumber\\
 & =(n+1)\int_{\mathbb{R}}p_{n}(c)c^{n}f(c)dc=(n+1)\int_{\mathbb{R}}p(n)^{2}f(c)dc\ne0
\end{align}
such that the closure is not Galilean invariant.

\section{Extended Gramian Closure\label{sec:extendedGram}: The Even Case}

It becomes clear that the Gramian moment closure must be extended
in order to accommodate full gauge-invariance while keeping hyperbolicity.
In the present section, we consider an extension of the Gramian closure in the even case,
that is, in the case where of up to an even order $2n$ are given, and the next (odd) moment $u_{2n+1}$ is to be predicted.

\subsection{Definition and Hyperbolicity\label{subsec:extGramHyper}}

Inspired by the closure proposed in \cite{FoxLaurent} the following
family of closures will be formulated. It contains the Gramian closure
(\ref{eq:gramian-closure}) for $\chi=0$.
\begin{defn}[Extended Gramian Closure -- Even]
Let $M\in\mathbb{N}$ be even and $n=\frac{M}{2}$. For an arbitrary
$\chi\in\mathbb{R}$ the condition
\begin{align}
\int_{\mathbb{R}}p_{n}(c)c^{n+1}f(c)dc\int_{\mathbb{R}}p_{n-1}(c)c^{n-1}f(c)dc & =\chi\int_{\mathbb{R}}p_{n}(c)c^{n}f(c)dc\int_{\mathbb{R}}p_{n-1}(c)c^{n}f(c)dc\label{eq:closure-relation-ext}
\end{align}
defines the relation 
\begin{align}
u_{2n+1} & =u_{n+1,2n}^{T}G_{n-1}^{-1}u_{n,2n-1}+\chi{\textstyle \frac{u_{2n}-u_{n,2n-1}^{T}G_{n-1}^{-1}u_{n,2n-1}}{u_{2n-2}-u_{n-1,2n-3}^{T}G_{n-2}^{-1}u_{n-1,2n-3}}}\left(u_{2n-1}-u_{n,2n-2}^{T}G_{n-2}^{-1}u_{n-1,2n-3}\right)\label{eq:gramian-closure-ext}
\end{align}
between $u_{M+1}=u_{2n+1}$ and $\left\{ u_{k}\right\} _{k=0,1,...,M}$,
which we call \emph{extended Gramian moment closure} for even $M$. 
\end{defn}
While in \cite{FoxLaurent} the closure was given only in an iterative,
algorithmic procedure, the integrals in (\ref{eq:closure-relation-ext})
can be easily evaluated on the explicit form of the orthogonal polynomials
given in (\ref{eq:pn-explicit}) and then solved for an explicit expression
for $u_{2n+1}$. The extended Gramian closure requires the solution
of two linear systems, one of size $\frac{M}{2}$ for the vector $G_{n-1}^{-1}u_{n,2n-1}$
and one of size $\frac{M}{2}-1$ for $G_{n-2}^{-1}u_{n-1,2n-3}$.
For $M=4$ we have $n=2$ and the explicit expression of the closure
is
\begin{align}
u_{5} & =\left(\begin{array}{c}
u_{3}\\
u_{4}
\end{array}\right)^{T}\left(\begin{array}{cc}
u_{0} & u_{1}\\
u_{1} & u_{2}
\end{array}\right)^{-1}\left(\begin{array}{c}
u_{2}\\
u_{3}
\end{array}\right)\nonumber \\
 & \quad\quad+\chi{\scriptstyle \frac{u_{3}-u_{2}u_{0}^{-1}u_{1}}{u_{2}-u_{1}u_{0}^{-1}u_{1}}}\left(u_{4}-\left(\begin{array}{c}
u_{2}\\
u_{3}
\end{array}\right)^{T}\left(\begin{array}{cc}
u_{0} & u_{1}\\
u_{1} & u_{2}
\end{array}\right)^{-1}\left(\begin{array}{c}
u_{2}\\
u_{3}
\end{array}\right)\right),
\end{align}
while for $M=6$, that is, $n=3$ we abbreviate 
\begin{align}
\left(\begin{array}{c}
b_{1}\\
b_{2}\\
b_{3}
\end{array}\right):=\left(\begin{array}{ccc}
u_{0} & u_{1} & u_{2}\\
u_{1} & u_{2} & u_{3}\\
u_{2} & u_{3} & u_{4}
\end{array}\right)^{-1}\left(\begin{array}{c}
u_{3}\\
u_{4}\\
u_{5}
\end{array}\right),\quad\quad & \left(\begin{array}{c}
c_{1}\\
c_{2}
\end{array}\right)=\left(\begin{array}{cc}
u_{0} & u_{1}\\
u_{1} & u_{2}
\end{array}\right)^{-1}\left(\begin{array}{c}
u_{2}\\
u_{3}
\end{array}\right)
\end{align}
and have 
\begin{align}
u_{7} & =\left(\begin{array}{c}
u_{4}\\
u_{5}\\
u_{6}
\end{array}\right)^{T}\left(\begin{array}{c}
b_{1}\\
b_{2}\\
b_{3}
\end{array}\right)+\chi\,{\scriptstyle \frac{u_{5}-\left(\begin{array}{c}
u_{3}\\
u_{4}
\end{array}\right)^{T}\left(\begin{array}{c}
c_{1}\\
c_{2}
\end{array}\right)}{u_{4}-\left(\begin{array}{c}
u_{2}\\
u_{3}
\end{array}\right)^{T}\left(\begin{array}{c}
c_{1}\\
c_{2}
\end{array}\right)}}\left(u_{6}-\left(\begin{array}{c}
u_{3}\\
u_{4}\\
u_{5}
\end{array}\right)^{T}\left(\begin{array}{c}
b_{1}\\
b_{2}\\
b_{3}
\end{array}\right)\right).
\end{align}
We note that the only requirement for the existence of the extended
Gramian closure of order $M=2n$ is the invertibility of $G_{n-1}$
and $G_{n-2}$. This is much less
restrictive than demanding full realizability of the moment vector
$\left\{ u_{k}\right\} _{k=0,1,...,M}$, that is, the positive-definiteness
of $G_{n}$.
\begin{thm}
The extended Gramian closure (\ref{eq:closure-relation-ext}) for
even $M$ with $n=\frac{M}{2}$ implies the characteristic polynomial
\begin{align}
\mathcal{P}_{M+1}^{(\text{\textnormal{ext}})}\left(z\right) & =p_{n}(z)\left(p_{n+1}(z)-\chi\frac{\sigma_{n,n}}{\sigma_{n-1,n-1}}p_{n-1}(z)\right),\label{eq:gram-charac-ext}
\end{align}
 where $p_{n+1}$ uses the moment $u_{2n+1}$ implied by the closure
(\ref{eq:gramian-closure-ext}). The definition of the expression $\sigma_{n,n}$ is given below.
\end{thm}
\begin{proof}
We will define the following expressions for a given $k\in\mathbb{N}$
\begin{align}
\sigma_{k,k}:= & \int_{\mathbb{R}}p_{k}(c)c^{k}f(c)dc=u_{2k}-u_{k,2k-1}^{T}G_{k-1}^{-1}u_{k,2k-1},\label{eq:sigma_n,n}\\
\sigma_{k,k+1}:= & \int_{\mathbb{R}}p_{k}(c)c^{k+1}f(c)dc=u_{2k+1}-u_{k+1,2k}^{T}G_{k-1}^{-1}u_{k,2k-1},\label{eq:sigma_n,n+1}
\end{align}
see also \cite{FoxLaurent,gautschi}. In the proof of Thm.~\ref{thm:The-Gramian-closure}
we found (\ref{eq:pn+1-explicit}) which can be written
\begin{align}
p_{k+1}(c) & =c^{k+1}-(1,c,\cdots, c^{k-1})G_{k-1}^{-1}u_{k+1,2k}-\frac{\sigma_{k,k+1}}{\sigma_{k,k}}p_{k}(c)\label{eq:pn+1-again}
\end{align}
using the above abbreviations. We also define the differential operator 
\begin{align}
\text{grad}_{M+1} & :=\left(\partial_{u_{0}},\partial_{u_{1}},\cdots,\partial_{u_{M+1}}\right)
\end{align}
and use the operation $(1,c,\cdots, c^{M+1})\cdot\text{grad}_{M+1}$ to write the characteristic polynomial for a general closure 
$\mathcal{C}(\cdot)$ as
\begin{align}
\mathcal{P}_{M+1}^{(\text{ext})}\left(c\right) &= (1,c,\cdots, c^{M+1})\cdot\text{grad}_{M+1}\left(u_{M+1}-\mathcal{C}(u_{0},u_{1},\cdots, u_{M})\right).
\label{eq:charac1}
\end{align}
Note, that the statement of Thm.~\ref{thm:The-Gramian-closure} can be rephrased
as 
\begin{align}
(1,c,\cdots, c^{M+1})\cdot\text{grad}_{M+1}\left(\sigma_{k,k+1}\right) & =\left(p_{k+1}(c)+\frac{\sigma_{k,k+1}}{\sigma_{k,k}}p_{k}(c)\right)p_{k}(c),\quad\quad\quad 0\le k\le\frac{M}{2}\label{eq:grad_sigma_nn+1}
\end{align}
where we used (\ref{eq:pn+1-again}). By copying the procedure in
Thm.~\ref{thm:The-Gramian-closure} we can similarly deduce the statement
\begin{align}
(1,c,\cdots, c^{M+1})\cdot\text{grad}_{M+1}\left(\sigma_{k,k}\right) & =p_{k}(c)p_{k}(c)\quad\quad\quad0\le k\le\frac{M}{2}\label{eq:grad_sigma_nn}
\end{align}
for the gradient of $\sigma_{k,k}$.

With the above relations for the gradients of $\sigma_{k,k+1}$ and $\sigma_{k,k}$ we now directly compute for the closure (\ref{eq:gramian-closure-ext}) the 
characteristic polynomial
\begin{align}
\mathcal{P}_{M+1}^{(\text{ext})}\left(c\right) &= (1,c,\cdots, c^{M+1})\cdot\text{grad}_{M+1}\left(\sigma_{n,n+1}-\chi\frac{\sigma_{n,n}}{\sigma_{n-1,n-1}}\sigma_{n-1,n}\right)\nonumber \\
 & =\left(\frac{\sigma_{n,n+1}}{\sigma_{n,n}}-\chi\frac{\sigma_{n-1,n}}{\sigma_{n-1,n-1}}\right)p_{n}(c)p_{n}(c)+\left(p_{n+1}(c)-\chi\frac{\sigma_{n,n}}{\sigma_{n-1,n-1}}p_{n-1}(c)\right)p_{n}(c)
\end{align}
after some cancellations and resorting. The first term vanishes due
to the closure relation (\ref{eq:closure-relation-ext}).
\end{proof}
\begin{cor}
\label{cor:ext-Gramian-strict}The extended Gramian closure (\ref{eq:gramian-closure-ext})
is strictly hyperbolic for all $\chi>-1$.
\end{cor}
\begin{proof}
All orthogonal polynomials $p_{n}$ and $p_{n+1}$ have interlacing
roots. We show in Appendix \ref{sec:Interlacing-Theorem}, that this
also holds true for $p_{n}$ and the modified polynomial 
\begin{align}
\widehat{p}_{n+1}(z) & =p_{n+1}(z)-\chi\frac{\sigma_{n,n}}{\sigma_{n-1,n-1}}p_{n-1}(z),
\end{align}
if $\chi>-1$. Hence, the characteristic polynomial $\mathcal{P}_{M+1}^{(\text{ext})}$
in (\ref{eq:gram-charac-ext}) gives strict hyperbolicity. 
\end{proof}

\subsection{Invariance}

In Sec.~\ref{subsec:InvarianceGram} we have seen the closure (\ref{eq:gramian-closure-ext})
with $\chi=0$ is not Galilean invariant. Below we show that there
is only one choice of $\chi$ which gives invariance for the extended
closure. This invariant closure is equivalent to the closure given
in \cite{FoxLaurent}.
\begin{thm}
The extended Gramian closure (\ref{eq:gramian-closure-ext}) with
the choice $\chi=\frac{n+1}{n}$ is fully gauge-invariant in the
sense of Def.~(\ref{def:gauge-invariance}).
\end{thm}
\begin{proof}
Following Theorem (\ref{thm:gauge-invariance}) we will check the
definitions (\ref{eq:invarianceConditions}) on the characteristic
polynomial (\ref{eq:gram-charac-ext}). The first condition directly
gives for any $\chi\in\mathbb{R}$ 
\begin{align}
\int_{\mathbb{R}}\mathcal{P}_{M+1}^{(\text{ext})}(c)f(c)dc & =\int_{\mathbb{R}}p_{n}(c)\left(p_{n+1}(c)-\chi\frac{\sigma_{n,n}}{\sigma_{n-1,n-1}}p_{n-1}(c)\right)f(c)dc=0
\end{align}
due to the orthogonality of the polynomials. For the second condition
we find after employing orthogonality
\begin{align}
\int_{\mathbb{R}}{\textstyle \frac{d}{dc}}\mathcal{P}_{M+1}^{(\text{ext})}(c)f(c)dc & =-\chi\frac{\sigma_{n,n}}{\sigma_{n-1,n-1}}\int_{\mathbb{R}}p_{n}'(c)p_{n-1}(c)f(c)dc+\int_{\mathbb{R}}p_{n}(c)p_{n+1}'(c)f(c)dc\nonumber\\
 & =-\chi\frac{\sigma_{n,n}}{\sigma_{n-1,n-1}}n\int_{\mathbb{R}}c^{n-1}p_{n-1}(c)f(c)dc+(n+1)\int_{\mathbb{R}}p_{n}(c)c^{n}f(c)dc
\end{align}
which gives zero only for $\chi=\frac{n+1}{n}$. Finally, the third
condition yields
\begin{align}
\int_{\mathbb{R}}c{\textstyle \frac{d}{dc}}\mathcal{P}_{M+1}^{(\text{ext})}(c)f(c)dc & =-\chi\frac{\sigma_{n,n}}{\sigma_{n-1,n-1}}\int_{\mathbb{R}}cp_{n-1}(c)p_{n}'(c)f(c)dc+\int_{\mathbb{R}}cp_{n}(c)p_{n+1}'(c)f(c)dc\label{eq:inv-cond3}
\end{align}
where we use the representation (\ref{eq:pn+1-again}) in the form
\begin{align}
p_{n+1}'(c) & =(n+1)c^{n}-\frac{\sigma_{n,n+1}}{\sigma_{n,n}}p_{n}'(c)+q_{n-1}'(c)
\end{align}
with some $q_{n-1}\in\mathbb{P}_{n-1}$. For the second integral we thus
obtain
\begin{align}
\int_{\mathbb{R}}cp_{n}(c)p_{n+1}'(c)f(c)dc & =(n+1)\int_{\mathbb{R}}c^{n+1}p_{n}(c)f(c)dc-\frac{\sigma_{n,n+1}}{\sigma_{n,n}}\int_{\mathbb{R}}cp_{n}(c)p_{n}'(c)f(c)dc\nonumber\\
 & =(n+1)\sigma_{n,n+1}-n\frac{\sigma_{n,n+1}}{\sigma_{n,n}}\int_{\mathbb{R}}c^{n}p_{n}(c)f(c)dc\nonumber\\
 & =\sigma_{n,n+1}.
\end{align}
The first integral in (\ref{eq:inv-cond3}) is structurally identical
and gives $\sigma_{n-1,n}$ such that 
\begin{align}
\int_{\mathbb{R}}c{\textstyle \frac{d}{dc}}\mathcal{P}_{M+1}^{(\text{ext})}(c)f(c)dc & =-\chi\frac{\sigma_{n,n}}{\sigma_{n-1,n-1}}\sigma_{n-1,n}+\sigma_{n,n+1}=0
\end{align}
because of the closure relation (\ref{eq:closure-relation-ext}) for
any value of $\chi$.
\end{proof}

\subsection{Equilibrium Preservation}
\begin{thm}
The extended Gramian closure (\ref{eq:gramian-closure-ext}) with $\chi=\frac{n+1}{n}$ is equilibrium-preserving for Gaussian equilibria.
\end{thm}

\begin{proof}
As seen from the previous section, the extended closure (with $\chi=\frac{n+1}{n}$) is gauge-invariant. Therefore, it is sufficient to consider the equilibrium distribution (\ref{eq:maxwellian}) with $\theta=1$, $\rho=1$, and $v=0$.
In this case, the moments are given by
\begin{align}
u_{n}^{(\text{eq})} & = \begin{cases}
  \frac{1}{\sqrt{\pi}} \Gamma\left(\frac{n+1}{2}\right), n\,\text{even} \\
  0, n\,\text{odd}.
\end{cases}
\end{align}
Thus, to prove equilibrium preservation in the even case, it suffices to prove that the closure defined by (\ref{eq:closure-relation-ext}) gives $u_{2n+1}=0$ in case of equilibrium for any $n$.

For the equilibrium distribution, which is a Gaussian distribution, the orthogonal polynomials $p_n$ are given by the probabilist's Hermite polynomials $\operatorname{He}_n(c)$. The Hermite polynomials are odd for odd $n$ and even for even $n$. Since the equilibrium distribution is symmetric, the following relation holds:
\begin{align}
\int_{\mathbb{R}}p_{n}(c)c^{m}f(c)dc	=\begin{cases}
0 & \text{if }n+m\text{ is odd}\\
\ne0 & \text{if }n+m\text{ is even}
\end{cases}
\end{align}
We have that for all $k=0,1,...n$ 
\begin{align}
  \int_{\mathbb{R}} p_{k-1}(c) c^{k} f(c) dc = 0\qquad\text{and}\qquad  \int_{\mathbb{R}} p_{k}(c) c^k f(c) dc \ne 0
\end{align}
because these integrals contain the given equilibrium moments $u^{(\text{eq})}_{0},...u^{(\text{eq})}_{2n}$ and not the unknown closure. Therefore, the closure relation (\ref{eq:closure-relation-ext}) reduces in equilibrium to
\begin{align}
  \int_{\mathbb{R}}p_{n}(c)c^{n+1}f(c)dc & =0,
\end{align}
which is an expression to compute the unknown moment $u_{2n+1}$.
Since $p_n$ is monic and either even or odd, we can write it as $p_n(c)=c^n + q_{n-2}(c)$, where $q_{n-2}$ is a polynomial of degree $n-2$ that is also even when $p_n$ is even and odd when $p_n$ is odd.
Therefore,
\begin{align}
  0 = \int_{\mathbb{R}}p_{n}(c)c^{n+1}f(c)dc &= \int_{\mathbb{R}}c^{2n+1}f(c)dc + \int_{\mathbb{R}}q_{n-2}(c)c^{n+1}f(c)dc \nonumber\\
  & = u_{2n+1} + \int_{\mathbb{R}}q_{n-2}(c)c^{n+1}f(c)dc = u_{2n+1},
\end{align}
since $q_{n-2}(c)c^{n+1}$ is an odd polynomial of degree $2n-1$, and the corresponding integral is equal to 0. We conclude that the closure indeed gives $u_{2n+1}=0$ for any $n$ in case of given equilibrium moments.
\end{proof}

\section{Extended Gramian Closure: The Odd Case\label{sec:The-Odd-Case}}

The definition (\ref{eq:sigma_n,n}) above shows that the condition $\sigma_{n,n}=0$
yields the equation
\begin{align}
u_{2n}&=u_{n,2n-1}^{T}G_{n-1}^{-1}u_{n,2n-1}
\label{eq:gram-closure-odd}
\end{align}
as closure relation for $u_{2n}$ in the case of odd $M$ with $n=\frac{M+1}{2}$, which corresponds to the QMOM closure~\cite{mcgraw1997description}. The finding in (\ref{eq:grad_sigma_nn}) 
for the gradient of $\sigma_{n,n}$ yields the characteristic polynomial of the closure as $p_{n}(z)p_{n}(z)$. It is easy
to show that this polynomial satisfies the invariance conditions (\ref{eq:invarianceConditions}),
hence, the closure for $u_{M+1}$ will be gauge-invariant as well.
However, the closure is not strictly hyperbolic. The characteristic
polynomial has $2n$ real-valued roots, but they will always be pair-wise identical for
all choices of realizable moments $\left\{ u_{k}\right\} _{k=0,1,...,M}$, see also the discussion in \cite{FoxLaurent} 

\subsection{Definition}
By following the approach of the closure for even $M$ it is possible
to derive a closure for odd $M$ which is gauge-invariant and avoids
double roots of the characteristic polynomial, at least in general. 
\begin{defn}[Extended Gramian Closure -- Odd]
Let $M\in\mathbb{N}$ be odd and $n=\frac{M+1}{2}$. For an arbitrary
$\chi\in\mathbb{R}$ the condition
\begin{align}
\int_{\mathbb{R}}p_{n-1}(c)c^{n+1}f(c)dc\int_{\mathbb{R}}p_{n-1}(c)c^{n-1}f(c)dc & =\chi\left(\int_{\mathbb{R}}p_{n-1}(c)c^{n}f(c)dc\right)^{2}\label{eq:closure-relation-ext-odd}
\end{align}
defines the relation 
\begin{align}
u_{2n} & =u_{n+1,2n-1}^{T}G_{n-2}^{-1}u_{n-1,2n-3}+\chi\frac{(u_{2n-1}-u_{n,2n-2}^{T}G_{n-2}^{-1}u_{n-1,2n-3})^{2}}{u_{2n-2}-u_{n-1,2n-3}^{T}G_{n-2}^{-1}u_{n-1,2n-3}}\label{eq:gramian-closure-ext-odd}
\end{align}
between $u_{M+1}=u_{2n}$ and $\left\{ u_{k}\right\} _{k=0,1,...,M}$,
which we call \emph{extended Gramian moment closure} for odd $M$.
\end{defn}
As before the explicit closure (\ref{eq:gramian-closure-ext-odd})
is readily computed from the closure relation and the explicit form
of the polynomials (\ref{eq:pn-explicit}). It only requires one linear
solve of a system of size $\frac{M-1}{2}$. 

\noindent \textbf{Examples:} For $M=3$ we have $n=2$ and $G_{n-1}=G_{0}=(u_{0})$
and 
\begin{align}
u_{4} & =\frac{u_{3}u_{1}}{u_{0}}+\chi\frac{(u_{3}-\frac{u_{2}u_{1}}{u_{0}})^{2}}{u_{2}-\frac{u_{1}u_{1}}{u_{0}}},\label{eq:example_odd_u4}
\end{align}
while for $M=5$ and $n=3$ we abbreviate 
\begin{align}
\left(\begin{array}{c}
b_{1}\\
b_{2}
\end{array}\right) & =\left(\begin{array}{cc}
u_{0} & u_{1}\\
u_{1} & u_{2}
\end{array}\right)^{-1}\left(\begin{array}{c}
u_{2}\\
u_{3}
\end{array}\right)
\end{align}
and the closure reads
\begin{align}
u_{6} & =\left(\begin{array}{c}
u_{4}\\
u_{5}
\end{array}\right)^{T}\left(\begin{array}{c}
b_{1}\\
b_{2}
\end{array}\right)+\chi\,{\scriptstyle \frac{\left(u_{5}-\left(\begin{array}{c}
u_{3}\\
u_{4}
\end{array}\right)^{T}\left(\begin{array}{c}
b_{1}\\
b_{2}
\end{array}\right)\right)^{2}}{u_{4}-\left(\begin{array}{c}
u_{2}\\
u_{3}
\end{array}\right)^{T}\left(\begin{array}{c}
b_{1}\\
b_{2}
\end{array}\right)}}.
\end{align}

We note that the only requirement for the existence of the extended
Gramian closure of order $M=2n-1$ is the invertibility of $G_{n-2}$.

\begin{thm}
The extended Gramian closure (\ref{eq:closure-relation-ext-odd})
for odd $M$ with $n=\frac{M+1}{2}$ implies the characteristic polynomial
\begin{align}
\mathcal{P}_{M+1}^{(\text{\textnormal{ext}})}\left(z\right) & =p_{n-1}(z)\left(\widetilde{p}_{n+1}(z)-2\chi\frac{\sigma_{n-1,n}}{\sigma_{n-1,n-1}}p_{n}(z)\right),\label{eq:gram-charac-ext-odd}
\end{align}
 where $\widetilde{p}_{n+1}$ is given by 
\begin{align}
\widetilde{p}_{n+1}(c) & =c^{n+1}-(1,c,\cdots, c^{n-1})G_{n-1}^{-1}u_{n+1,2n}\label{eq:p_n+1-tilde}
\end{align}
which uses the moment $u_{2n}$ implied by the closure (\ref{eq:gramian-closure-ext}). 
\end{thm}
Using (\ref{eq:pn+1-again}) the polynomial $\widetilde{p}_{n+1}$ can be formally related to $p_{n+1}$
by
\begin{align}
\widetilde{p}_{n+1}(c) & =p_{n+1}(c)+\frac{\sigma_{n,n+1}}{\sigma_{n,n}}p_{n}(c)\label{eq:p_n+1-tilde-again}
\end{align}
where both $p_{n+1}$ and $\sigma_{n,n+1}$ require the moment of
degree $2n+1=M+2$. Those are not available in a moment theory of
odd degree $M$ and its closure. However, the representation of $\widetilde{p}_{n+1}$
in (\ref{eq:p_n+1-tilde}) shows that the polynomial is computable using the closure. Equation
(\ref{eq:p_n+1-tilde-again}) will be used in orthogonality arguments.
\begin{proof}
We build onto the findings of the previous sections, in particular the representation (\ref{eq:charac1}) and take the derivative of the closure in the form
\begin{align}
\sigma_{n-1,n+1}-\chi\frac{\sigma_{n-1,n}^{2}}{\sigma_{n-1,n-1}} & =0
\label{eq:extoddclosure2}
\end{align}
where we introduced for $k\in\mathbb{N}$
\begin{align}
\sigma_{k-1,k+1}:= & \int_{\mathbb{R}}p_{k-1}(c)c^{k+1}f(c)dc=u_{2k}-u_{k+1,2k-1}^{T}G_{k-2}^{-1}u_{k-1,2k-3}
\end{align}
while $\sigma_{k-1,k}$ and $\sigma_{k-1,k-1}$ are defined in (\ref{eq:sigma_n,n+1})/(\ref{eq:sigma_n,n}).

We find
\begin{align}
(1,c,\cdots, c^{M+1})\cdot\text{grad}_{M+1}\left(\sigma_{n-1,n+1}\right)=\quad\quad\quad\quad\\
\left(c^{n+1}-(1,c,\cdots, c^{n-2})G_{n-2}^{-1}u_{n+1,2n-1}\right) & \left(c^{n-1}-(1,c,\cdots, c^{n-2})G_{n-2}^{-1}u_{n-1,2n-3}\right)
\end{align}
using the same procedure as before. We identify the
second factor as $p_{n-1}(c)$. For the first we consider the explicit
formula for $p_{n+1}$ in (\ref{eq:pn+1-explicit}) and apply the
Schur complement formula (\ref{eq:schurbilinear}) a second time to
the bilinear form of $G_{n-1}$. This yields
\begin{align}
\widetilde{p}_{n+1}(c) & =c^{n+1}-(1,c,\cdots,c^{n-1})G_{n-1}^{-1}u_{n+1,2n}\nonumber \\
 & =c^{n+1}-(1,c,\cdots, c^{n-2})G_{n-2}^{-1}u_{n+1,2n-1}-\frac{\sigma_{n-1,n+1}}{\sigma_{n-1,n-1}}p_{n-1}(c)\label{oddpn+1}
\end{align}
and generalizing the derivation for $0\le k\le\frac{M}{2}$ we identify 
\begin{align}
(1,c,\cdots, c^{M+1})\cdot\text{grad}_{M+1}\left(\sigma_{k-1,k+1}\right) & =\left(\widetilde{p}_{k+1}(c)+\frac{\sigma_{k-1,k+1}}{\sigma_{k-1,k-1}}p_{k-1}(c)\right)p_{k-1}(c)
\end{align}
which supplements the expressions (\ref{eq:grad_sigma_nn+1}) and (\ref{eq:grad_sigma_nn}).

Using this we can compute the characteristic polynomial
\begin{align}
\mathcal{P}_{M+1}^{(\text{ext})}\left(c\right)  &=(1,c,\cdots, c^{M+1})\cdot\text{grad}_{M+1}\left(\sigma_{n-1,n+1}-\chi\frac{\sigma_{n-1,n}^{2}}{\sigma_{n-1,n-1}}\right) \\
&= p_{n-1}(c)\left(\widetilde{p}_{n+1}(c)-2\chi\frac{\sigma_{n-1,n}}{\sigma_{n-1,n-1}}p_{n}(c)+\left(\frac{\sigma_{n-1,n+1}\sigma_{n-1,n-1}-\chi\sigma_{n-1,n}^{2}}{\sigma_{n-1,n-1}^{2}}\right)p_{n-1}(c)\right).\nonumber 
\end{align}
After invoking the closure relation (\ref{eq:closure-relation-ext-odd}) in the form (\ref{eq:extoddclosure2})
we arrive at the characteristic polynomial (\ref{eq:gram-charac-ext-odd}). 
\end{proof}

Using similar arguments as in Corollary \ref{cor:ext-Gramian-strict} it is possible to show that both factors in the characteristic polynomial (\ref{eq:gram-charac-ext-odd}) have real-valued roots and the closure is hyperbolic. 
Unfortunately, it is not possible to show that this characteristic polynomial will always have distinct roots and, hence, the closure can not be strictly hyperbolic in general. However, for the equilibrium case, empirical numerical investigations with values of $M$ up to 13 show that the roots of the characteristic polynomial are indeed distinct for $n=(M+1)/2$ odd and only the root zero is double in the case of even $n$. Special choices of moments in the phase space may generate more multiple roots of the characteristic polynomial, but the closure is essentially strictly hyperbolic, at least close to equilibrium and for odd $n$. A detailed analysis of the situation and the investigation of different closures for odd $M$ is left for future work.

\subsection{Invariance and Equilibrium Preservation}
First, we consider the gauge-invariance of the closure and prove the following statement.
\begin{thm}
	The extended Gramian closure (\ref{eq:gram-charac-ext-odd}) with
	the choice $\chi=\frac{n+1}{2n}$ is fully gauge-invariant in the
	sense of Def.~(\ref{def:gauge-invariance}).
\end{thm}
\begin{proof}
We follow the same procedure and check the
invariance conditions on the characteristic polynomial. The first condition holds
	due to orthogonality of the polynomials. Using the representation of $\widetilde{p}_{n+1}$ from $ (\ref{oddpn+1}),$ we identify the second condition is satisfied only if $\chi=\frac{n+1}{2n}.$
Following the same procedure in the proof of Theorem 6.2, the third condition becomes
	\begin{align}
		\int_{\mathbb{R}}c{\textstyle \frac{d}{dc}}\mathcal{P}_{M+1}^{(\text{ext})}(c)f(c)dc  
		=2\sigma_{n-1,n+1}-2\chi \frac{\sigma_{n-1,n}^{2}}{\sigma_{n-1,n-1}}
	\end{align}
and this is equal to zero because of the closure relation (\ref{eq:closure-relation-ext}) for any value of $\chi$.
\end{proof}

For equilibrium preservation we have a negative result. 

\begin{observation}
	The extended Gramian closure (\ref{eq:gram-charac-ext-odd}) is not equilibrium-preserving.
\end{observation}

This can easily be seen from the example~(\ref{eq:example_odd_u4}) for the case of $M=3$: since for equilibrium $u_1=0$, $u_3=0$, it follows from (\ref{eq:example_odd_u4}) that $u_4=0$, which is obviously not the correct value for a Gaussian distribution. The situation is similar for the Gramian closure (\ref{eq:gram-closure-odd}).

\subsection{Summary}
\begin{table}[t]
	\centering
	\resizebox{1\textwidth}{!}{ 
	\begin{tabular}{|M{0.2\textwidth}|M{0.15\textwidth}|M{0.25\textwidth}|M{0.15\textwidth}|M{0.25\textwidth}|}
		\cline{2-5}
		\multicolumn{1}{c|}{} & \multicolumn{2}{c|}{\textbf{Even case}} & \multicolumn{2}{c|}{\textbf{Odd case}} \\ \cline{2-5}
		\multicolumn{1}{c|}{} & \textbf{Gramian} & \textbf{Extended Gramian  } & \textbf{Gramian} & \textbf{Extended Gramian} \\ \hline
		\textbf{Strict hyperbolicity} & 
		\cmark & \cmark\hspace{0.02cm} for all $\chi > -1 $  & \xmark$^{\textcolor{ForestGreen}{**}}$ & \cmark$^{\textcolor{RedOrange}{ **} }$\\ \hline
		\textbf{Gauge invariance} & 
		\xmark & \cmark\hspace{0.02cm} with $\chi=\frac{n+1}{n}$ & \cmark & \cmark\hspace{0.02cm} with $\chi=\frac{n+1}{2n}$ \\ \hline
		\textbf{Equilibrium preservation} & 
		\cmark & \cmark & \xmark & \xmark \\ \hline
		\multicolumn{5}{|c|}{\xmark$^{\textcolor{ForestGreen}{**}} :=$ weak hyperbolicity,\quad \cmark$^{\textcolor{RedOrange}{**}} :=$ close to equilibrium \rule{0pt}{0.5cm}} \\ \hline
	\end{tabular}}
\caption{Overview of the different Gramian closures for both even and odd order $M$. In both cases a simple closure exists (\ref{eq:gramian-closure})/(\ref{eq:gram-closure-odd}) and an extended version (\ref{eq:gramian-closure-ext})/(\ref{eq:gramian-closure-ext-odd}). Each closure has different preservation properties. }
	\label{table:closures}
\end{table}

We have derived 4 different closures, namely a simple and an extended closure for each case of even and odd order $M$. Each closure has different preservation properties. An overview is displayed in Table \ref{table:closures}. It is interesting to see that the three properties considered here are largely independent. Gauge-invariance can be achieved for special parameter choice for all except the first closure. Strict hyperbolicity comes in different levels of which the extreme cases are: always distinct roots and always pair-wise identical roots. The last closure is in between as occasional double or multiple roots can not be excluded, and there are situations with fully distinct roots.  

If the equilibrium distribution is an even function, like the Maxwellian (\ref{eq:maxwellian}) for $v=0$, then equilibrium preservation is easy for the even order closures. In that case they only need to produce a zero moment as equilibrium closure. In the odd case the next moment is even and will have a specific value depending on the form of the equilibrium distribution. However, equilibrium is a concept outside of the theory of orthogonal polynomials and it is not surprising that both odd case closures can not preserve equilibrium. 

The table and considerations so far do not tell how well the next moment is actually approximated by the closure. While it is difficult to find a general estimate of the approximation quality, the error in the prediction can be studied empirically for concrete cases of distribution functions.

\section{Numerical Case Studies\label{sec:numerical}}

To test the prediction quality of the Gramian and extended Gramian closures for the $(M+1)^{th}$ moment, given the set of $M$ lower moments, we work on three test problems based on different families of distribution functions. We consider a case based on the description of shock waves in gas dynamics, as given by the well-known Mott-Smith distribution \cite{cite_6}. Another problem is chosen from plasma dynamics and uses electron-hole distributions within a traveling wave \cite{bujarbarua1981theory}. Lastly, one case is designed based on a strong bimodal distribution to analyze the behavior of the new closures near the realizability boundary. The extended Gramian closures use the values of $\chi$ such that they are gauge-invariant. These are  $\chi=\frac{n+1}{n}$ for the case of even $M$ with $n=\frac{M}{2}$, and $\chi=\frac{n+1}{2n}$ for the odd case with $n=\frac{M+1}{2}$.

\subsection{Other Closure Techniques}

We compare the moment prediction with that of well-known closure techniques. One is constructed by H. Grad \cite{grad1949kinetic} as a perturbation of the equilibrium distribution using a Hermite polynomial expansion. Another approach \cite{levermore1996moment,kogan1969equations,WDreyer1987} is centered on entropy maximization and provides a nonlinear closure that maximizes entropy under the constraints of given moments. 

\subsubsection{Grad's Closure}

The classical moment expansion technique \cite{grad1949kinetic} has been developed as an approach to close the system (\ref{eq:momentSystem}) by using an expansion of the distribution function as
\begin{equation}
f^{(\text{Grad})}(c)=f^{(\text{eq})}(c)\sum_{j=0}^{M}\alpha_{j} \text{He}_{j}(c),
\end{equation}
where $\alpha_{j=0,\dots,M}$ are the coefficients of the ansatz, and $ \text{He}_j$ are Hermite polynomials. The closure is constructed as a perturbation of the equilibrium distribution with coefficients ${\alpha_{j}}$ that satisfy
\begin{equation}
 u_{k}=\int c^{k}f^{(\text{Grad})}(c)dc\hspace{0.1em},\hspace{0.2em}\text{for}\hspace{0.25em} k=0,\dots,M.
 \label{eq:gradmoments}
\end{equation}
This is a linear system for the $\alpha_{j}$ values, whose solution is a mapping between the $\alpha_{j}$ values and given moments. 
The usage of the Grad closure has a long history and many successful results exist \cite{torrilhon2016modeling, StruchBook}. Still, it is known to lose hyperbolicity, particularly at strong non-equilibrium states when using a local equilibrium in the expansion which guarantees gauge-invariance. In our computations, this is achieved
by taking dimensionless moments scaled by density and temperature and shifted by velocity in \eqref{eq:gradmoments} before calculating the distribution and back-transforming the next higher moment. 

A closer look also shows that the Grad distribution is not strictly positive for the large velocity values. However, the Maxwellian term in $f^{(\text{eq})}(c)$ ensures that values for the distribution function for large velocities will be very small. Overall, the Grad closure is expected to work better when used close to equilibrium.
\begin{figure}[t]
	\begin{center}
		\includegraphics[width=\columnwidth]{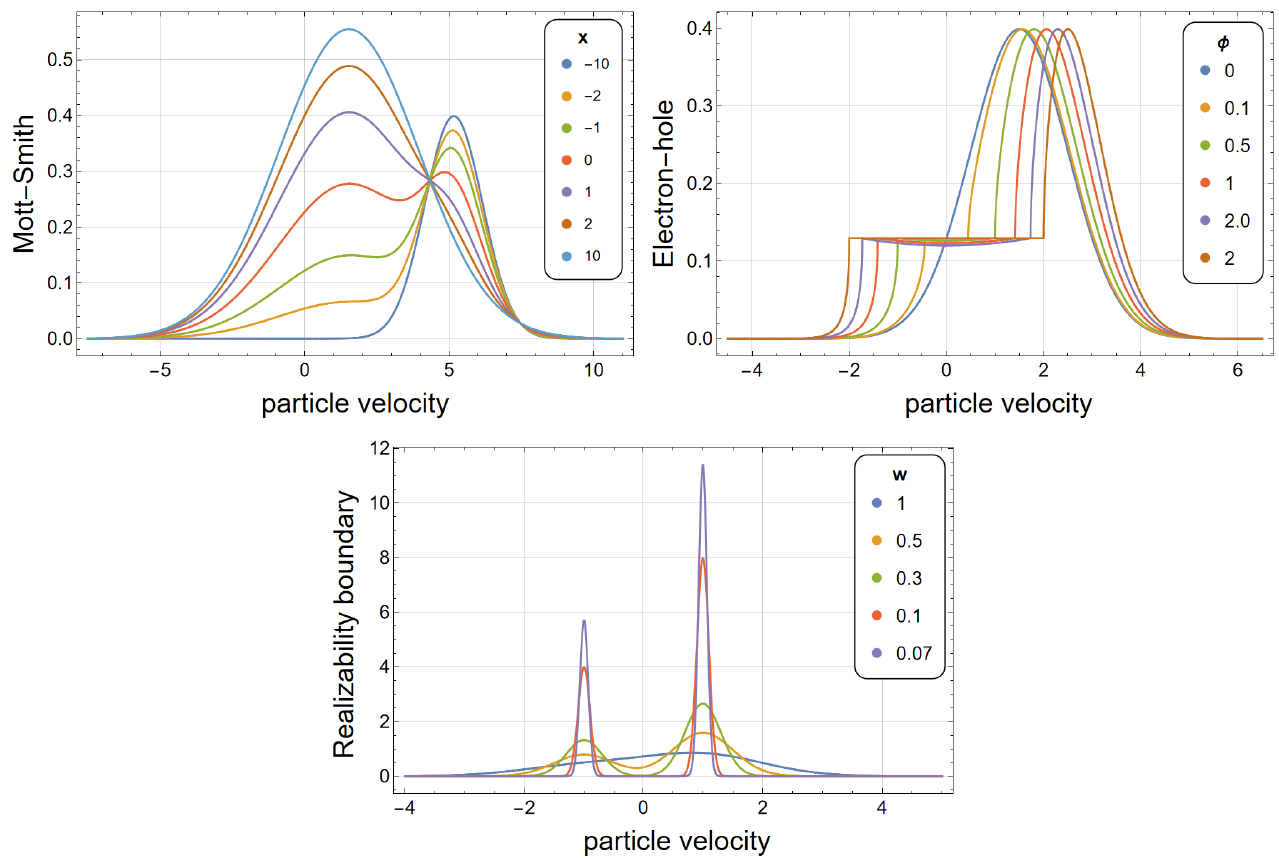}
	\end{center}
	\caption{Family of distribution functions for test problems in subsection  \ref{subsec:mottsmith}, \ref{subsec:electronhole}, and \ref{subsec:bimodal}. The Mott-Smith distribution (a) for shock waves \cite{cite_6} is set up as a bimodal type for shock waves and uses parameters related by the Rankine-Hugoniot conditions. The electron-hole distribution (b) from \cite{bujarbarua1981theory} with $\beta=-0.05$ is a piece-wise function, one part is a shifted Gaussian and the other one is close to a flat-top distribution.  The distribution function in (c) is modeled by a superposition of two Gaussian to analyze the closures near the realizability boundary.
	Apart from some fixed parameters, the change in the distributions is with respect to (a): position, (b): potential, (c): width of two Gaussian. Some example distributions from these families are displayed.  }
	\label{fig:HD}
\end{figure}

\subsubsection{Maximum Entropy Closure}

The maximum entropy closure offers an alternative technique based on designing a distribution function that maximizes the Boltzmann entropy $-f\log f$ and fits all moments considered. This can be defined as an optimization problem 
\begin{equation}
	f^{(\text{ME})}(c)=\operatorname*{\mathbf{argmax}}_{f:\mathbb{R}\to \mathbb{R}^+}
		\left(-\int_{\mathbb{R}}f\log f\hspace{0.1em} dc\right)\\\vspace{0.1em}\quad \text{s.t.}\quad\hspace{0.15em} u_{k}=\int_{\mathbb{R}}c^{k}f^{(\text{ME})}\hspace{0.1em}dc \hspace{0.1em},\hspace{0.2em}\text{for}\hspace{0.25em} k=0,\dots,M.
		\label{MeOpti}
\end{equation} 
The optimization problem is typically solved by Lagrange multipliers ${\lambda_{j=0,1,...M}}$ and the multipliers are chosen such that the given moment constraints are satisfied. If the maximum entropy distribution can be computed, the closure will be hyperbolic\cite{levermore1996moment}. However, for order $M$ larger than two, solving for Lagrange multipliers is not possible analytically and comes with excessive computational cost. Also, well-posedness of the optimization problem also can not be easily guaranteed \cite{Junk}.

For the maximum entropy approach, we do not solve the primal problem (\ref{MeOpti}) for the Lagrange multipliers, but directly solve for point values of $\bm{f}\in\mathbb{R}^{N_{c}}$ on a discrete velocity grid $\{c_{j}\}_{j=1,...,N_{c}}$ on a finite domain $[c_{\text{min}},c_{\text{max}}]$. The optimization problem becomes
\begin{align}
	\bm{f}^{(\text{ME})}=\operatorname*{\mathbf{argmax}}_{\bm{f}\in\mathbb{R}^{N_{c}}}\left(-\sum_{j=1}^{N_{c}}f_{j}\log f_{j}\right)\quad&\text{s.t.}\quad u_{k}=\sum_{j=1}^{N_{c}}(c_{j})^{k}f_{j},\quad\text{for}\,k=0,...M	
\end{align}
and can be solved for the vector $\bm{f}^{(\text{ME})}$ in a straight forward way. To facilitate the convergence the logarithm has been regularized by 
writing $-f\log(f+\varepsilon)$ where $\varepsilon=10^{-8}$.  The finite domain $[c_{\text{min}},c_{\text{max}}]$ and the number of quadrature points $N_c$ are chosen specifically for each numerical test problem below after some experimenting to achieve stable results for all investigated orders of moment $M$.

Note that this approach also allows to produce an approximation of the distribution in case of odd $M$ where the maximum entropy distribution is usually not integrable. We will display these predictions for odd $M$ under the label maximum-entropy below.

\subsubsection{Error Computation}

In the studies below we define a hidden truth which is a distribution function compatible with the corresponding test cases, see Fig.~\ref{fig:HD}. The true moments $\left\{ u^{(\text{hidden})}_{k}\right\} _{k=0,1,...,M+1}$ are calculated with this distribution function $f^{\text{(hidden)}}(c)$ for a certain choice of $M$, including the next moment $M+1$.
For each of the different closure, the next moment $u_{\text{closure}}=u_{M+1}$ is calculated from the moments provided up to $M$ by
\begin{equation}
	u_{\text{closure}}=\int c^{M+1}f^{\text{(recons)}}(c)dc = \mathcal{C}\left(u_{0},u_{1},\cdots, u_{M}\right)
\end{equation}
either using a reconstructed distribution, like in the Grad and maximum entropy closures, or in a direct approach as for all the Gramian closures. 

Since an explicit distribution is not always reconstructed we assess the predictive quality solely relying on the value of the next higher moment. We consider the relative error which is the difference between $u^{(\text{hidden})}_{M+1}$ from the hidden truth and the prediction of the next moment $u_{\text{closure}}$ from the closure method and it is defined by
\begin{align}
	e_\text{rel}(u_{\text{closure}})=\left|\frac{u_{\text{closure}} - u^{(\text{hidden})}_{M+1}}{u^{(\text{hidden})}_{M+1}}\right|.
\end{align}
The error will depend on the underlying distribution and the number of moments $M$.

\subsection{Mott-Smith Distribution}
\begin{figure}[t]
	\begin{center}
		\includegraphics[width=\columnwidth]{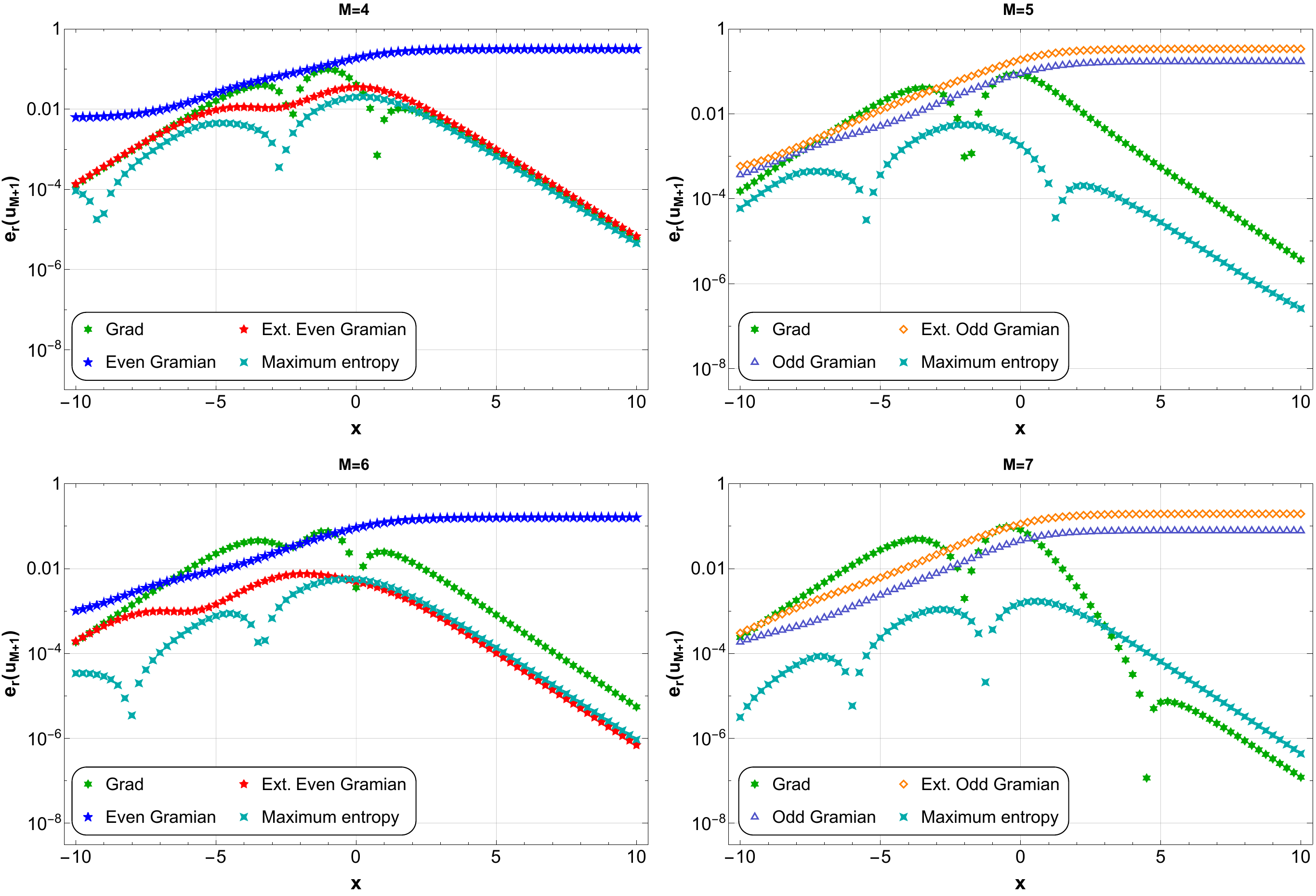}
	\end{center}
	\caption{Relative error $e_r(u_{\text{closure}})$ of the next higher moment calculated for different number of moments using different closure techniques is shown for the Mott-Smith shock wave distribution \cite{cite_6}.  We consider $\text{Ma}=4$, $\gamma=5/3$ and compute moments with the distribution $f_{\text{MS}}$ for each $x$ values from -10 to 10 with step size 0.25. For the maximum entropy closure, we discretize velocity domain $[-6,9]$ with 1000 points. 
Note, that the Gramian and extended Gramian closure are defined differently and marked by different color in the 
plots. }
	\label{fig:MS}
\end{figure}
\label{subsec:mottsmith}
The Mott-Smith distribution function\cite{cite_6} is based on a bimodal superposition of Maxwellian type distributions
\begin{equation}
	f_{\text{BM}}(\rho_{1},v_{1},\theta_{1},\rho_{2},v_{2},\theta_{2},c)=\dfrac{\rho_{1}}{\sqrt{2\pi\theta_{1}}} \exp\left(\dfrac{-(c-v_{1})^2}{2\theta_{1}}\right)+ \dfrac{\rho_{2}}{\sqrt{2\pi\theta_{2}}} \exp\left(\dfrac{-(c-v_{2})^2}{2\theta_{2}}\right),
	\label{eq:Bimodal distr.}
\end{equation}
with densities ($\rho_{i})$, velocities ($v_{i}$), and temperatures $\theta_{i}$ for $i=1,2$. Mott-Smith interpolates between the equilibrium values before and after a shock with Mach number $\text{Ma}$ continuously along a distance $x$ and uses the distribution function in the form
\begin{equation}
	f_{\text{MS}}(\text{Ma},x,c,\gamma)=f_{\text{BM}}\left(\frac{1}{1+\exp(x)}, \text{Ma}\sqrt{\gamma},1,
	\left(1- \frac{1}{1+\exp(x)}\right)\rho_{*},v_{*}\sqrt{\gamma},\theta_{*},c\right)
\end{equation} 
with the Rankine-Hugoniot values $\rho_{*}=({\text{Ma}}^2(\gamma+1))/(2+{M_a}^2(\gamma-1)),$ $v_{*}=(\text{Ma})/\rho_{*},$ $\theta_{*}=(1-\gamma+ 2\gamma {\text{Ma}}^2)/(1+\gamma)\rho_{*}$. The adiabatic coefficient is given by $\gamma$. Fig.~\ref{fig:HD}, part (a), shows how the distribution function changes as the gas passes through a shock wave along the spatial variable $x$. Before the shock, the gas is at equilibrium with a narrow distribution function, then it undergoes compression and heating as it passes through the shock wave, reflected in a broadened and skewed distribution function, and finally emerges in a new equilibrium state with a wider and shifted distribution function at a higher temperature. 

Figure \ref{fig:MS} shows the closure errors along the spatial coordinate $x$ through the shock wave. Before and after the shock we are close to equilibrium, and the Grad and maximum entropy closures, as well as the extended even order Gramian, exhibit relative errors that go to zero. Note that several of the error curves show strong spikes indicating very small errors. These singularities occur at very different positions and vary strongly with the choice of parameters and test cases. At these points the error crosses through zero momentarily as we vary the parameter which results in a spike
in the logarithmic plot of the modulus value. Being difficult to explain, we consider these situations to be outliers and focus on the general trend in the error curves. 

Among the closures by maximum entropy gives clearly the best over-all performance. The Grad closure is also very good for these cases, but cannot be used in practice due to lack of hyperbolicity. In the Gramian family, the extended even-order closure works very well. Note that the other Gramian closures all have global hyperbolicity and the bad approximation performance thereof must be attributed to the lack of other structure-preserving properties, i.e. gauge invariance and equilibrium preservation.

\begin{figure}[t]
	\begin{center}
		\includegraphics[width=\columnwidth]{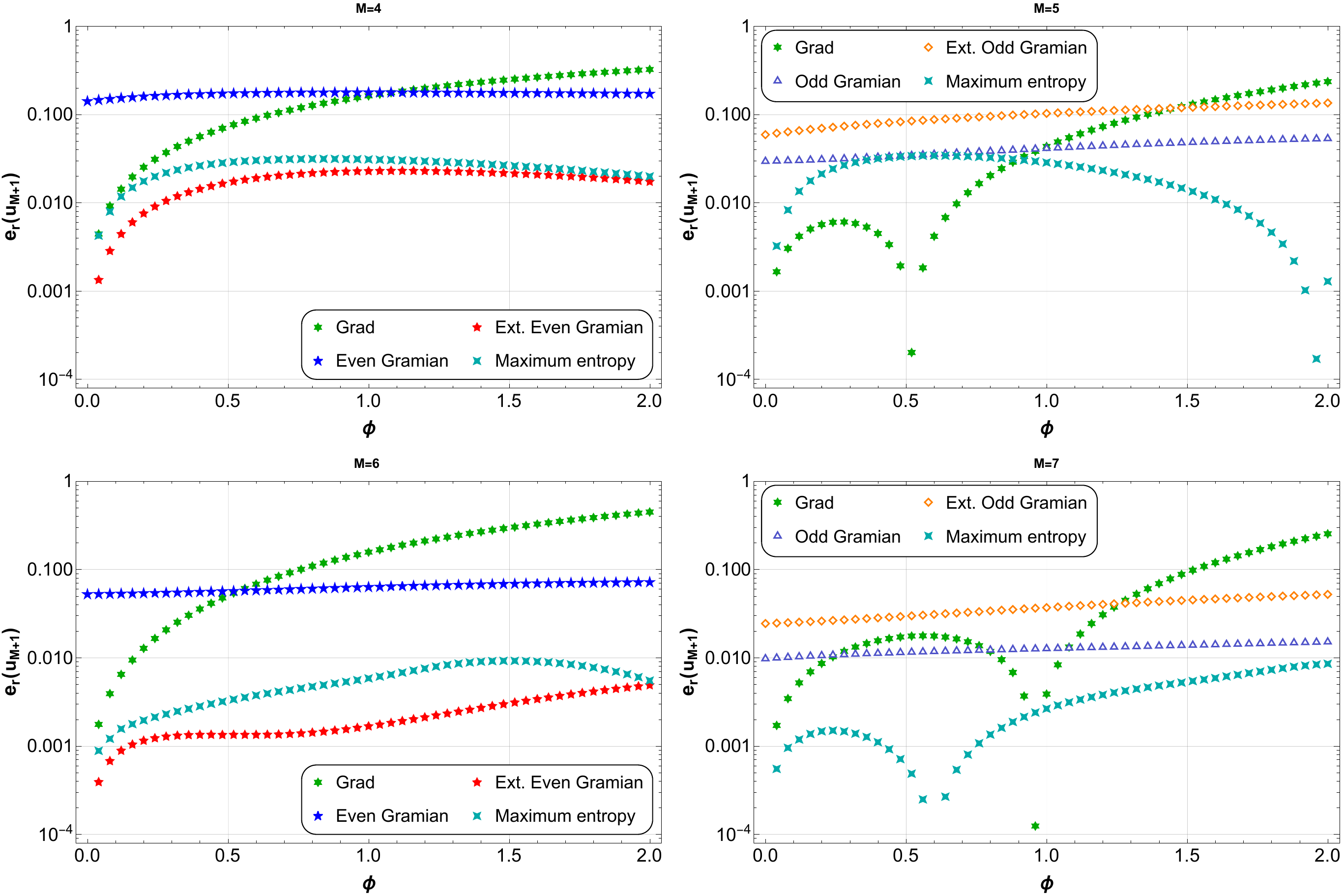}
	\end{center}
	\caption{Relative error between $u_{\text{closure}}$ and $u_{M+1}$
		with fixed parameters $v_{0}=1.5$, $\beta=-0.05$ for the electron hole distribution \cite{bujarbarua1981theory} with
		different electrostatic potentials $0\leq\phi\leq2$ with step size $0.04$. Even (left) and odd (right) case scenarios were examined and are shown separately due to differences in the definition of the Gramian and extended Gramian closures. For the maximum entropy closure, the velocity domain $[-6,8]$ is discretized with 1000 grid points.}
	\label{fig:EH}
\end{figure}

\subsection{Electron-Hole Distribution}
\label{subsec:electronhole}
The distribution function for the electron-hole test problem is defined in \cite{bujarbarua1981theory} and reads 
\begin{equation}
	f_{\text{EH}}(v,\phi,v_{0},\beta)=\begin{cases}
		(2\pi)^{-1/2}\exp\left[-\frac{1}{2}\left(\sign(v)(v^{2}-2\phi)^{1/2}-v_{0}\right)^{2}\right]&, v^{2}>2\phi\\
		(2\pi)^{-1/2}\exp\left[-\frac{1}{2}\left(\beta(v^{2}-2\phi)+{v_{0}}^{2}\right)\right]&, v^{2}<2\phi
	\end{cases}
\end{equation}
where $\phi$ is the electrostatic potential, $\beta$ is a measure for the number of trapped electrons, and $v_{0}$ is the frame velocity. The first part describes the distribution of the untrapped electrons and other is that of the trapped electrons. The effect of the electrostatic potential in the shape of distribution function $f_{\text{EH}} $ can be seen in the part (b) of the Fig.~\ref{fig:HD}. When the electrostatic potential is zero, the distribution function is a shifted Maxwellian, however, a change in the potential $\phi$ increases the deviation from equilibrium of the distribution. 

Fig.~\ref{fig:EH} shows the relative errors of even and odd order closures based on Grad, maximum entropy as well as the Gramian closures. With a potential close to zero the distribution shape gets similar to an equilibrium distribution and the Grad and maximum entropy closures give very small errors. As above, error singularities can be observed in several results. Among the Gramian closures, the extended closure for even order (left column of Fig.~\ref{fig:EH}) clearly stands out among other closures showing errors of similar high quality as maximum entropy. On the right side of the figure, the odd Gramian closure performs better than the extended odd version independently from number of moments. While the bad error in the odd order case can be explained by the missing preservation of equilibrium, it is interesting to note, that the simple Gramian closure in the even order case also shows bad performance, though the equilibrium is preserved. Here, the strong deviation must be associated with the missing gauge-invariance.
\subsection{Realizability Boundary}\label{sec:realizboundary}
\label{subsec:bimodal}
\begin{figure}[t]
	\begin{center}
		\includegraphics[width=1.01\columnwidth]{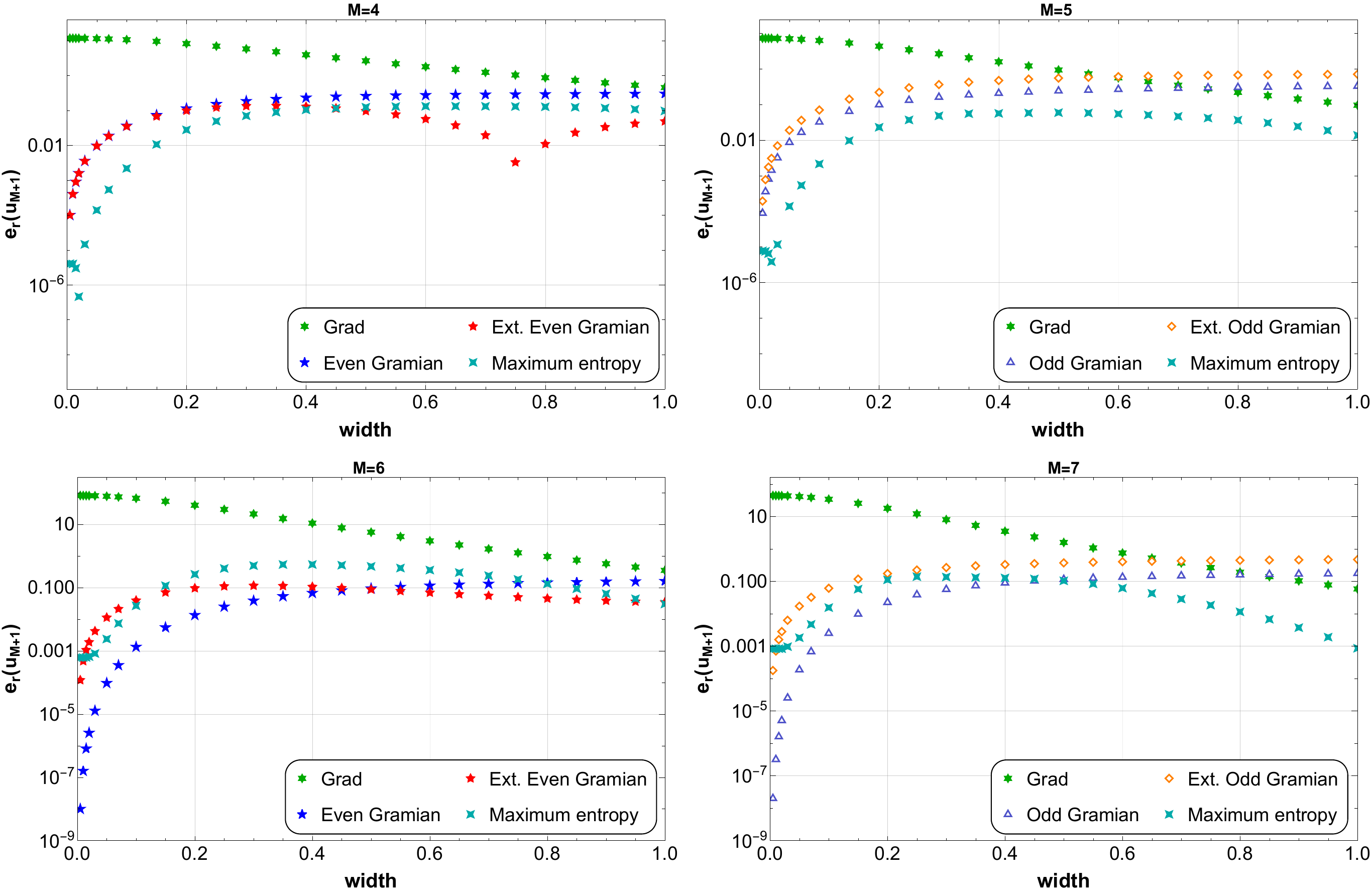}
	\end{center}
		\caption{Relative error $e_r(u_{\text{closure}})$ of the next higher moment for a range of widths $w$ in the test of the realizability boundaries. The left part of the figure shows the relative error for even numbers of moments $M=4,6$ while the right one shows odd cases $M=5,7$. For the maximum entropy closure, computational domain is taken with $c\in[-4,5]$ with 1000 grid points.
	}
	\label{fig:BM}
\end{figure}

On the realizability boundary the distribution becomes a discrete measure given by a superposition of Dirac-deltas \cite{Curto1991,Aheizer1962}. If the distribution is given by $m$ separate Diracs, the Gramian matrix $G_{n}$ in \eqref{eq:gram-matrix} will be singular for $n\ge m$. In these cases the Gramian closures will formally fail due to singular matrices, for both even closures with $M\ge2m+2$, in the odd case with $M\ge 2m+1$, and for the extended odd closures with $M\ge 2m+3$. 

We will investigate this by considering a two Maxwellian-type distribution that gradually degenerates into two separate Dirac delta-functions. The bimodal distribution \eqref{eq:Bimodal distr.} is chosen with parameters $\rho_{1}=1, v_{1}=-1, \sqrt{\theta_{1}}=\sqrt{\theta_{2}}=w, \rho_{2}=2, v_{2}=1$, where $w$ is a measure for the width of the peaks of the distribution. The resulting distributions are shown in part (c) of Fig.~\ref{fig:HD}. For $w \to 0$ the distribution will approach the realizability boundary. The different closures are computed for $w$ values between $0.1$ and $1$ with step size of $0.05$, as well as for a set of smaller values of $w$: $\{0.07, 0.05, 0.03, 0.02, 0.015, 0.01, 0.005\}$.

The relative error of the next moment is plotted in Fig.~\ref{fig:BM}. As expected, the Grad closure has difficulties with strongly bimodal distributions and shows small error only for less bimodal distributions. The maximum entropy closure is formally able to represent a number of Diracs, however, the current numerical method has difficulties in finding a valid approximation. As a consequence the errors stagnate for maximum-entropy for very small $w$. However, overall, the maximum entropy approach mostly provides very good relative errors and is comparable with the Gramian closure methods. 

Remarkably, the new Gramian closures perform well even near the realizability boundary and also for moment numbers beyond the threshold discussed above. This suggests that the evaluation of the closure stays continuous for large $M$ also when approaching the realizability boundary. A proper limit expression could be obtained by analyzing this singularity further.

\begin{figure}[t]
	\begin{center}
		\includegraphics[width=0.75\columnwidth]{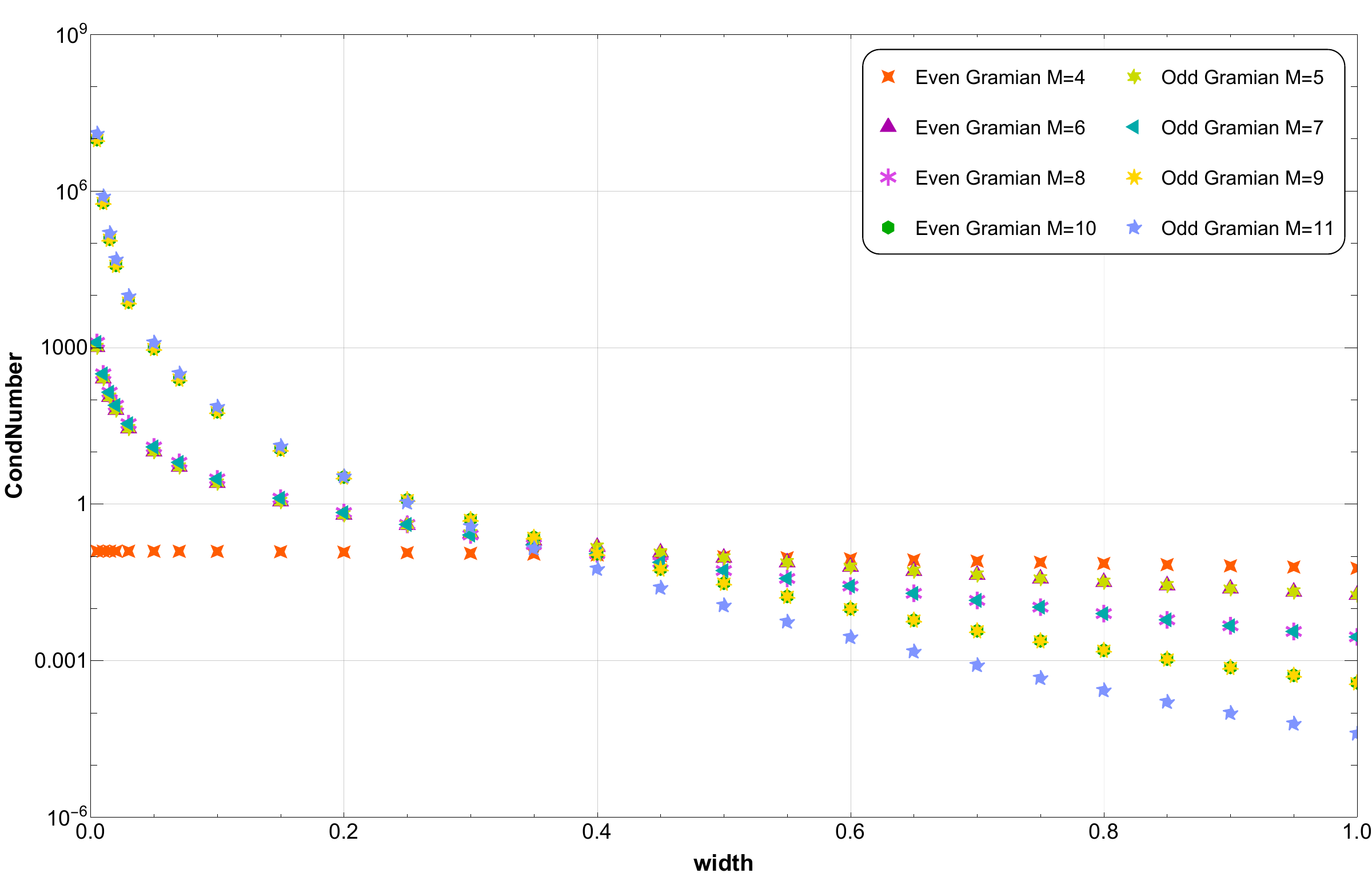}
	\end{center}
	\caption{Condition numbers of different closure methods as a function of the order of the moment theory for the bimodal test problem of subsection \ref{subsec:bimodal}. The x-axis represents the width of the distribution function, and the y-axis displays the corresponding condition number of the Gramian matrix. Each colored line represents a different closure method. Only even and odd Gramian closure results are shown in the figure. The condition numbers of the extended Gramian closures are quite similar to the Gramian closures. Therefore, these results are excluded from the figure. }
	\label{fig:BMAQ}
\end{figure}

The stability of the evaluation of the Gramian closures is affected by the condition numbers of the involved matrices. Figure \ref{fig:BMAQ} shows how the width $w$ of the bimodal distribution affects the condition number for various moment orders. The figure is restricted to the non-extended even and odd Gramian closure relations. For large width values, the distribution function is closer to equilibrium and a decrease in the condition number can be observed as the moment order is increased. In contrast, as the width approaches zero, the distribution function becomes strongly bimodal, and the condition number increases significantly, as expected. In these extreme cases, the condition numbers for $M=5,6,7,8$ and $M=9,10,11$ are nearly identical. Condition numbers in the range of $10^6$ do not strongly affect the numerical evaluation of the closures, especially as the involved matrix dimensions are small. 

\begin{figure}[t]
	\begin{center}
		\includegraphics[width=\columnwidth]{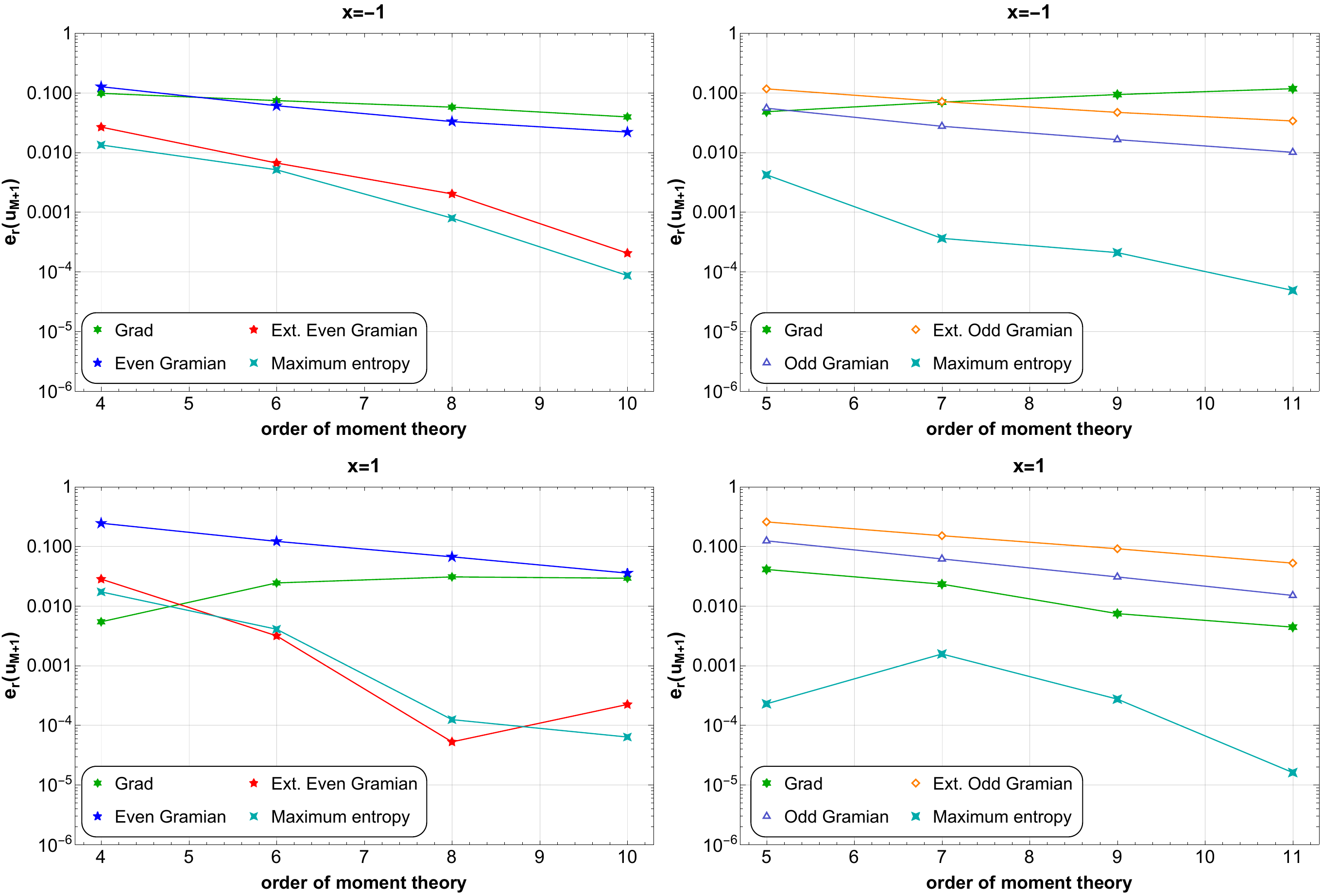}
	\end{center}
	\caption{The relative errors of the next higher moment for the Mott-Smith distribution from subsection \ref{subsec:mottsmith}. The x-axis corresponds to the increasing order of the moment theory, whereas the y-axis depicts the corresponding relative error. As above even and odd cases are shown left and right. Two different scenarios are presented for comparison. The top row shows the result $x = -1$ (strong non-equilibrium), while the bottom row shows the result $x = 1$ (weak non-equilibrium). Each colored line in the figure corresponds to a specific closure, using the identical labeling as in Figure \ref{fig:MS}.}
	\label{fig:MSEAQ}
\end{figure}

\subsection{Empirical Approximation Quality}
Both the closure of Grad and the maximum entropy approach rely on a reconstruction of the underlying distribution from the given set of moments. Increasing the number $M$ provides more information and it is reasonable to expect an ever better approximation of the distribution and consequently the next higher moment $u_{M+1}$. Even though the Gramian closures do not reconstruct a distribution it does make sense to ask about the behavior of the relative error of the next higher moment with the expectation that more given moments improve the accuracy of the closure prediction. 

In this section we study the approximation quality of the different closures by selecting specific distributions and compute the closures with $M=4,6,8,10$ for even, and $M=5,7,9,11$ for odd orders, respectively. The specific distributions are taken from Sec.~\ref{subsec:mottsmith} and Sec.~\ref{subsec:electronhole} above, and the relative errors are displayed in Figs.~\ref{fig:MSEAQ} and \ref{fig:EHEAQ}. We consider:
\begin{itemize}\renewcommand{\itemsep}{0mm}
\item the Mott-Smith distribution at position $x=-1.0$, a strong non-equilibrium situation, in the upper row of Fig.~\ref{fig:MSEAQ}.
\item the Mott-Smith distribution at position $x=1.0$, which is closer to a Maxwellian, in the lower row of Fig.~\ref{fig:MSEAQ}.
\item the electron-hole-distribution with potential $\phi=0.2$, a mild deviation from equilibrium, in the upper row of Fig.~\ref{fig:EHEAQ}.
\item the electron-hole-distribution with potential $\phi=0.2$, a strongly deformed distribution, in the lower row of Fig.~\ref{fig:EHEAQ}.
\end{itemize}
Each row shows the relative error of the next higher moment with increasing order $M$ of the moment theory separately for the even and odd cases of $M$. The Grad closure exhibits the most difficulties to achieve better approximations for larger moment numbers and its error levels are generally high. Only for distributions close to equilibrium an error decay with more moments is visible for the Grad closure. The maximum entropy closure gives for all cases of distributions very low level as well as decay of the relative error with increasing moment order. Interestingly, the error decay of maximum entropy is not monotone in most of the cases. As discussed before, the Grad closure does not provide hyperbolic equations, and while the maximum entropy does, it is computationally very expensive.

The Gramian closures consistently show a decay of the relative error of the closure with increasing order of moments. The closures for odd values of $M$ show relatively large errors, but still exhibit decay for all distributions, even though they lack the preservation of equilibrium. Interestingly, the pure odd Gramian closure, which is only nonstrictly hyperbolic, has better errors than the extended odd closure with improved hyperbolicity. 

The errors of the extended Gramian closure for even values of $M$ show a striking similarity to the results of the maximum entropy approach. Both show very low errors and a non-monotone decay behavior. This is observable for weak and strong non-equilibrium. Consequently, this Gramian closure is a strong contender to be used as efficient replacement of the maximum entropy approach. It is also possible that both closures have a stronger theoretical connection than currently known.

\begin{figure}[t]
	\begin{center}
		\includegraphics[width=\columnwidth]{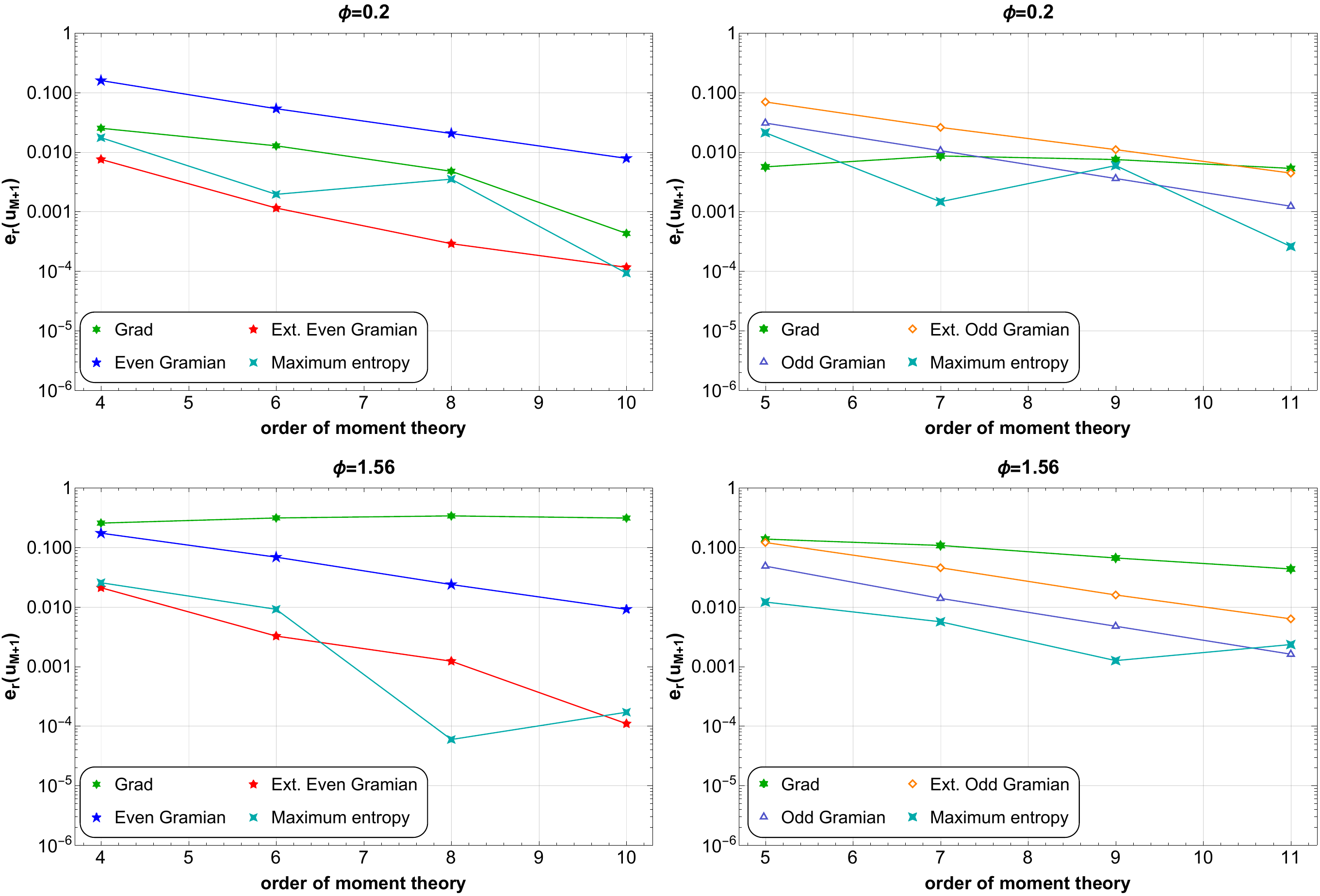}
	\end{center}
	\caption{The relative error of different closure methods for the electron-hole test problem in subsection \ref{subsec:electronhole} plotted against the order of the moment theory. Along the x-axis the order increases, while the y-axis displays the corresponding relative error. Left and right show even and odd order theories, respectively. The relative error results are presented for two different values of electrostatic potential $\theta$, namely 0.2 (top row) and 1.56 (bottom row). Each colored line in the figure corresponds to a different closure method, using the same markers as in Figure \ref{fig:EH}.}
	\label{fig:EHEAQ}
\end{figure}

\section{Conclusion}

This paper presented a family of nonlinear moment closures based on orthogonal polynomials which are tailored to the distribution function that the moments are based on. Using the ideas of \cite{FoxLaurent} as a starting point we formulated explicit expressions for the closures based on Gramian matrices and introduced an underlying theory of global hyperbolicity and gauge-invariance for the new family of closures. The last section investigated error behavior of the closures empirically and demonstrated very good approximation quality on a par with nonlinear closures like the maximum entropy approach. Furthermore, even near the realizability boundary, the prediction of such closures remains accurate. 

The investigation of these closures is just at the beginning. The similarities of the extended Gramian closure for even order theories with the maximum entropy approach, including very similar error behavior patterns, must be studied in detail. The presented theory could also be used to optimize further Gramian closures for the odd order case where an explicit equilibrium must be preserved. 

Eventually, these closures must be used to solve the evolution of kinetic equations. Some 1-dimensional simulation results for the solution of 1d kinetic equations have been shown in \cite{FoxLaurent}. A multi-dimensional extension remains open, but must be found to use these closures in realistic applications like super-sonic rarefied gas flow based on the Boltzmann equations.

\section*{Acknowledgments}
This work has been supported by the German Research Foundation within the research unit DFG-FOR5409. 

\appendix
\section{Schur-Complement\label{sec:Schur}}

We use the Schur complement to simplify the determinants of Gramian
matrices in Sec.~\ref{subsec:OrthPoly}. It also serves to reduce
bilinear forms in the representation of the orthogonal polynomials.
\begin{thm}
Consider a block-structured matrix 
\begin{align}
\left(\begin{array}{cc}
A & B\\
C^{T} & D
\end{array}\right) & \in\mathbb{R}^{(k+l)\times(k+l)}
\end{align}
build from the sub-matrices $A\in\mathbb{R}^{k\times k}$, $B,C\in\mathbb{R}^{k\times l}$,
and $D\in\mathbb{R}^{l\times l}$. The so-called Schur-complement
$S\in\mathbb{R}^{l\times l}$ is given by 
\begin{align}
S & =D-C^{T}A^{-1}B\label{eq:schur}
\end{align}
if the inverse of $A$ exists. If the Schur-complement is also invertible,
the inverse of the block-structured matrix can be written as
\begin{align}
\left(\begin{array}{cc}
A & B\\
C^{T} & D
\end{array}\right)^{-1} & =\left(\begin{array}{cc}
I_{k} & -A^{-1}B\\
0 & I_{l}
\end{array}\right)\left(\begin{array}{cc}
A^{-1} & 0\\
0 & S^{-1}
\end{array}\right)\left(\begin{array}{cc}
I_{k} & 0\\
-C^{T}A^{-1} & I_{l}
\end{array}\right)\label{eq:schurinv}
\end{align}
where $I_{k}$ and $I_{l}$ are identity matrices on $\mathbb{R}^{k}$
and $\mathbb{R}^{l}$. 
\end{thm}
This can be easily proved by computing the identity (\ref{eq:schurinv}).
The relevance of the Schur complement is extensively discussed, e.g.,
in the text book \cite{LinAlgBook}. In particular, we have the representation
\begin{align}
\det\left(\begin{array}{cc}
A & B\\
C^{T} & D
\end{array}\right) & =\det A\det\left(D-C^{T}A^{-1}B\right)\label{eq:schurdet}
\end{align}
for the determinant which is especially helpful in the case of vectorial
$B,C$, such that the Schur complement is scalar. Additionally, for
two vectors $v,w\in\mathbb{R}^{k+l}$ with
\begin{align}
v=\left(\begin{array}{c}
v_{k}\\
v_{l}
\end{array}\right),\quad & w=\left(\begin{array}{c}
w_{k}\\
w_{l}
\end{array}\right)
\end{align}
with $v_{k},w_{k}\in\mathbb{R}^{k}$ and $v_{l},w_{l}\in\mathbb{R}^{l}$
we obtain
\begin{align}
\left(\begin{array}{c}
v_{k}\\
v_{l}
\end{array}\right)^{T}\left(\begin{array}{cc}
A & B\\
C^{T} & D
\end{array}\right)^{-1}\left(\begin{array}{c}
w_{k}\\
w_{l}
\end{array}\right) & =v_{k}^{T}A^{-1}w_{k}+(v_{l}-v_{k}^{T}A^{-1}B)^{T}S^{-1}(w_{l}-C^{T}A^{-1}w_{k})\label{eq:schurbilinear}
\end{align}
for the bilinear form of the inverse.

\section{Interlacing Theorem\label{sec:Interlacing-Theorem}}

This appendix briefly discusses the interlacing property of the roots
of orthogonal polynomials relevant for strict hyperbolicity, e.g.,
in Sec.~\ref{subsec:extGramHyper}.
\begin{thm}[Cauchy Interlacing or Poincare Separation Theorem]
Consider the symmetric matrix
\begin{align}
A & =\left(\begin{array}{cc}
B & b\\
b^{T} & c
\end{array}\right)\in\mathbb{R}^{n\times n}
\end{align}
with the principal symmetric block $B\in\mathbb{R}^{(n-1)\times(n-1)}$,
sub-block $b\in\mathbb{R}^{n-1}$ and $c\in\mathbb{R}.$ Let $B$
have distinct eigenvalues $\lambda_{1}<\lambda_{2}<\cdots\lambda_{n-1}$
and no eigenvector orthogonal to $b$. Then the eigenvalues $\{\mu_{j}\}_{j=1,...,n}$
of $A$ satisfy 
\begin{align}
\mu_{1}<\lambda_{1}<\mu_{2}<\cdots & <\mu_{n-1}<\lambda_{n-1}<\mu_{n}
\end{align}
such that $\{\lambda_{i}\}$ interlaces the set $\{\mu_{j}\}$.
\end{thm}
For a more general statement and the proof, see also \cite{LinAlgBook}.
The case $b=0$ generates equality between the eigenvalues of $B$
and $n-1$ eigenvalues of $A$ different from $c$. 

The orthogonal polynomials $p_{k}$ in (\ref{eq:pn-explicit}) satisfy
a three-term recursion 
\begin{align}
p_{k+1}(z) & =(z-\alpha_{k})p_{k}(z)-\beta_{k}p_{k-1}(c),\quad\quad k=1,2,\ldots\label{eq:three-term-rec}
\end{align}
with $p_{0}(z)=1$ and $p_{-1}(z)=0$ and coefficients $\alpha_{k},\beta_{k}$
which follow from integrations of $p_{k}$ and $p_{k-1}$, see, e.g.,
\cite{gautschi,SzegoeOrthogonal}. It is also well-known that the
$n+1$ roots of the polynomial $p_{n+1}$ can be computed as eigenvalues
of the symmetric, tri-diagonal matrix
\begin{align}
J_{n} & =\left(\begin{array}{cccc}
\alpha_{0} & \beta_{1}^{1/2} &  & 0\\
\beta_{1}^{1/2} & \alpha_{1} & \ddots\\
 & \ddots & \ddots & \beta_{n}^{1/2}\\
0 &  & \beta_{n}^{1/2} & \alpha_{n}
\end{array}\right)\in\mathbb{R}^{(n+1)\times(n+1)}\label{eq:three-term-mat}
\end{align}
and that this matrix satisfies the conditions of the theorem above
such that the roots of $p_{n}$ and $p_{n+1}$ are interlacing. In
particular, with the notation of the interlacing theorem the submatrix
$B=J_{n-1}$ has no eigenvector orthogonal to $b=(0,\cdots,0,\beta_{n}^{1/2})^{T}$,
due to the tri-diagonal structure of $B$. 

We now replace the polynomial $p_{n+1}$ by 
\begin{align}
\widehat{p}_{n+1}(z)= & p_{n+1}(z)-\chi\frac{\sigma_{n,n}}{\sigma_{n-1,n-1}}p_{n-1}(z)
\end{align}
which is the relevant factor in the characteristic polynomial (\ref{eq:gram-charac-ext})
of the extended Gramian closure. The recursion (\ref{eq:three-term-rec})
is then modified for the $(n+1)$-th polynomial to 
\begin{align}
\widehat{p}_{n+1}(z) & =(z-\alpha_{n})p_{n}(z)-\widehat{\beta}_{n}p_{n-1}(c)
\end{align}
with
\begin{align}
\widehat{\beta}_{n} & =\beta_{n}+\chi\frac{\sigma_{n,n}}{\sigma_{n-1,n-1}}=(1+\chi)\frac{\sigma_{n,n}}{\sigma_{n-1,n-1}}
\end{align}
where we used the definition of the recursion coefficient $\beta_{n}$
from \cite{gautschi}. We conclude that the discussion of the eigenvalues
given by the matrix (\ref{eq:three-term-mat}) above still holds with
$\widehat{\beta}_{n}$, and the polynomials $p_{n}$ and $\widehat{p}_{n+1}$
will have interlacing roots for $\chi>-1$. This proves Corollary
\ref{cor:ext-Gramian-strict}.

\bibliographystyle{siam}
\bibliography{refBib}

\end{document}